\documentclass[a4paper,bibliography=totoc,12pt]{scrartcl}

\usepackage[english]{babel}
\usepackage[utf8]{inputenc}
\usepackage[T1]{fontenc}
\usepackage{lmodern}
\usepackage{textcomp}
\usepackage{amsfonts}
\usepackage{mathtools}
\usepackage{amssymb}
\usepackage[thmmarks,hyperref,amsmath]{ntheorem}
\usepackage{needspace}
\usepackage{csquotes}
\usepackage{siunitx}
\usepackage[shortlabels]{enumitem}
\usepackage[pdftex]{hyperref}
\usepackage{url}
\usepackage{xcolor}
\usepackage{tabularray}
\UseTblrLibrary{amsmath}
\UseTblrLibrary{booktabs}

\usepackage[
	backend=biber,
	style=numeric,
	giveninits = true,
	sorting = nyt,
	maxnames = 8,
	sortlocale = en_US,
	doi=true,
	isbn=false,
	date=year
]{biblatex}
\AtBeginBibliography{\small}
\DeclareFieldFormat{doi}{%
	\newline
	\mkbibacro{DOI}\addcolon\space
	\ifhyperref
		{\href{http://dx.doi.org/#1}{\nolinkurl{#1}}}
		{\nolinkurl{#1}}}
\DeclareFieldFormat{url}{%
	\newline
	\mkbibacro{URL}\addcolon\space
	\ifhyperref
		{\href{#1}{#1}}
		{\nolinkurl{#1}}}

\addtokomafont{disposition}{\rmfamily}
\addtokomafont{title}{}

\setlist[description]{
	font={\rmfamily},  
}
\setlist{itemsep=2pt}
\setlist{parsep=3pt}

\bibliography{bib_abrv.bib, degree_d_functions.bib}

\newcommand{\R}{\mathbb{R}}

\newcommand{\N}{\mathbb{N}}

\newcommand{\vek}[1]{\boldsymbol{#1}}

\DeclareMathOperator{\supp}{supp}

\DeclareMathOperator{\im}{im}

\DeclareMathOperator{\Der}{Der}
\DeclareMathOperator{\Res}{Res}

\theoremheaderfont{\bfseries}
\theorembodyfont{\upshape}
\theoremstyle{plain}
\theoremseparator{.}
\theoremsymbol{\ensuremath{\diamond}}
\newtheorem{theorem}{Theorem}
\newtheorem{corollary}[theorem]{Corollary}
\newtheorem{conjecture}[theorem]{Conjecture}
\newtheorem{lemma}{Lemma}[section]
\newtheorem{fact}[lemma]{Fact}

\newtheorem{definition}[lemma]{Definition}
\newtheorem{remark}[lemma]{Remark}

\newtheorem{example}[lemma]{Example}

\theoremheaderfont{\itshape}
\theorembodyfont{\upshape}
\theoremstyle{nonumberplain}
\theoremseparator{.}
\theoremsymbol{\rule{1.2ex}{1.2ex}}
\newtheorem{proof}{Proof}

\begin{document}
\title{The paired construction\\for Boolean functions on the slice}
\author{
Michael Kiermaier%
\thanks{
University of Bayreuth, Institute for Mathematics, 95440 Bayreuth, Germany
\newline
email:~\texttt{michael.kiermaier@uni-bayreuth.de}
\newline
homepage:~\url{https://mathe2.uni-bayreuth.de/michaelk/}
}
\and
Jonathan Mannaert%
\thanks{
Vrije Universiteit Brussel, Department of Mathematics and Data Science, Pleinlaan 2, B--1050 Brussels,
Belgium
\newline
email:~\texttt{Jonathan.Mannaert@vub.be}
}
\and
Alfred Wassermann%
\thanks{
University of Bayreuth, Institute for Mathematics, 95440 Bayreuth, Germany
\newline
email:~\texttt{alfred.wassermann@uni-bayreuth.de}
}
}
\maketitle

\begin{abstract}
	Let $V$ be a finite set of size $n$.
	We consider real functions on the \emph{slice} $\binom{V}{k}$, which are also known as functions in the Johnson scheme.
	For $I \subseteq J \subseteq V$, the characteristic function of the set of all $K\in\binom{V}{k}$ with $I \subseteq K \subseteq J$ is called \emph{basic}.
	In this article, we investigate a construction arising as the sum of two \enquote{opposite} basic functions.
	In essentially all cases, these \emph{paired} functions are Boolean.

	Our main result is the determination of the exact degree -- regarding a representation by an $n$-variable polynomial -- of all paired functions.
	The proof is elementary and does not involve any spectral methods.
	First, we settle the middle layer case $n=2k$ by identifying and combining various relations among the degrees involved.
	Then the general case is reduced to the middle layer situation by means of derived, reduced, and dual functions.

	Remarkably, in certain situations, the degree is strictly smaller than what is guaranteed by the elementary upper bound for the sum of functions.
	This makes paired functions good candidates for fixed-degree Boolean functions of small support size.
	As it turns out, for $n = 2k$ and even degree~$t \notin \{0,k\}$, paired functions provide the smallest known non-zero Boolean functions, surpassing the $t$-pencils, which is the smallest known construction in all other cases.
\end{abstract}


\section{Introduction}\label{sect:intro}

For a set $A$ and an integer $m$, the symbol $\binom{A}{m}$ denotes the set of all subsets of $A$ of size $m$.
The notation is suggested by the fact that $\#\binom{A}{m} = \binom{\#A}{m}$.
In the context of complexity theory, the set $\binom{A}{m}$ is known as a \emph{slice} (of the subset lattice of $A$).
The complement of $A$ in $V$ will be denoted by $A^\complement = V\setminus A$.
For a set $\mathcal{A}$ of sets we use the symbols
\[
	\bigcap \mathcal{A} = \bigcap_{A\in \mathcal{A}} A
	\qquad\text{and}\qquad
	\bigcup\mathcal{A} = \bigcup_{A\in \mathcal{A}} A\text{.}
\]

Throughout the article, we fix a set $V$ of finite size $n$ and a number $k\in\{0,\ldots,n\}$.
We consider functions $f : \binom{V}{k} \to \R$, which are known as (real) functions on the slice, or, when $2k \leq n$, as functions in the \emph{Johnson scheme} $J(n,k)$.
The function $f$ is called \emph{Boolean} if $\im(f) \subseteq \{0,1\}$.
As usual, the \emph{support} of $f$ is
\[
	\supp(f) = f^{-1}(\R \setminus\{0\}) = \{K\in\tbinom{V}{k} \mid f(a) \neq 0\}\text{,}
\]
and the characteristic function of a subset $\mathcal{K}\subseteq\binom{V}{k}$ is denoted by
\[
	\chi_{\mathcal{K}} : \binom{V}{k} \to \R\text{,}\quad K \mapsto \begin{cases} 1 & \text{if }K\in \mathcal{K}\text{,}\\0 & \text{if }K\notin\mathcal{K}\text{.}\end{cases}
\]
The mappings $f \mapsto \supp(f)$ and $\mathcal{K} \mapsto \chi_{\mathcal{K}}$ form an inverse pair of bijective functions between the set of all Boolean functions on $\binom{V}{k}$ and the set of all subsets of $\binom{V}{k}$, allowing us to silently identify these two kinds of objects with each other.
The elements of $\mathcal{K}$ will be called \emph{blocks}.

The zero function and the all-one function $\binom{V}{k} \to \R$ will be denoted by $\vek{0}_{\binom{V}{k}} = \vek{0}$ and $\vek{1}_{\binom{V}{k}} = \vek{1}$, respectively.
Clearly, these two functions are Boolean with $\vek{0} = \chi_{\emptyset}$ and $\vek{1} = \chi_{\binom{V}{k}}$.
The \emph{size} of $f$ is
\[
	\#f = \sum_{K\in\binom{V}{k}} f(K)\text{,}
\]
motivated by $\#f = \#\supp(f)$ for Boolean functions $f$.

For $I \subseteq J\subseteq V$, we define the \emph{basic set}
\[
	\mathcal{F}^{(V,k)}_{\!I,J} = \mathcal{F}_{\!I,J} = \{K\in\tbinom{V}{k} \mid I \subseteq K \subseteq J\}\text{.}
\]
Its characteristic function, denoted by
\[
	f^{(V,k)}_{I,J} = f_{I,J} = \chi_{\mathcal{F}_{I,J}}\text{,}
\]
will be called \emph{basic}, too.
The case $J=V$ is known as a \emph{pencil} or \emph{$(\#I)$-pencil} focussed at $A$.
Following~\cite{Kiermaier-Mannaert-Wassermann-2025-JCTSA212:P105979} in the case $q=1$, we define the \emph{degree} $\deg_V(f) = \deg(f)$ of a non-zero function $f : \binom{V}{k}\to\R$ as the smallest number $t$ such that $f$ is an $\R$-linear combination of $t$-pencils.
The degree of the zero function is set to $\deg(\vek{0}) = -\infty$.
It is known that $\deg(f) \leq \min(k,n-k)$, and we have $\deg(f) = 0$ if and only if $f \neq \vek{0}$ is constant.

\newpage
We will need the following results on the degree:
\begin{fact}[{{\cite[Lem.~4.3]{Kiermaier-Mannaert-Wassermann-2025-JCTSA212:P105979}}}]\label{fct:deg_linear_combination}
	Let $f,g : \binom{V}{k} \to \R$ and $\lambda\in \R$.
	Then
	\begin{enumerate}[(a)]
		\item\label{fct:deg_linear_combination:scalar_mul} $\deg(\lambda f) = \begin{cases} -\infty & \text{if }\lambda = 0\text{,}\\\deg(f) & \text{otherwise.}\end{cases}$
		\item\label{fct:deg_linear_combination:sum} $\deg(f \pm g) \leq \max(\deg(f), \deg(g))$, with equality whenever $\deg(f) \neq \deg(g)$.
	\end{enumerate}
\end{fact}

\begin{fact}[{{\cite[Th.~6.7]{Kiermaier-Mannaert-Wassermann-2025-JCTSA212:P105979}}}]\label{fct:basic_deg}
Let $I \subseteq J \subseteq V$ and define $i = \#I$ and $j = \#J$.
Then
\[
		\deg \mathcal{F}^{(V,k)}_{\!I,J}
		= \begin{cases}
			-\infty & \text{if } i > k\text{ or } j < k\text{,} \\
			\min(i+(n-j),k,n-k) & \text{otherwise.}
		\end{cases}
\]
\end{fact}

Each multivariate polynomial $\alpha \in \R[X_a \mid a\in V]$ represents a function $f_{\alpha} : \binom{V}{k} \to \R$, where the value $f_{\alpha}(K)$ is given by the evaluation of $\alpha$ at $X_a = 1$ if $a\in K$ and $X_a = 0$ if $a\notin K$, for all $a\in V$.
As an example, $f^{(V,k)}_{\!I,J}$ is represented by
\[
	\prod_{a\in I} X_a\; \cdot\; \prod_{b\in J^\complement}(1-X_b)\text{.}
\]
The evaluation map provides a surjective ring homomorphism $\R[X_a \mid a\in V] \to \R^{\binom{V}{k}}$.
Hence, denoting its kernel -- i.e. the set of all polynomials in $\R[X_a \mid a\in V]$ representing $\vek{0}$ -- by $\mathcal{I}$, the set of all functions $\binom{V}{k} \to \R$ can be identified with the quotient ring $\R[X_a \mid a\in V] / \mathcal{I}$.

\begin{fact}
	Let $f : \binom{V}{k} \to \R$.
	Then $\deg_V(f)$ equals the minimum degree of a polynomial $\alpha\in \R[X_a \mid a\in V]$ representing $f$, i.\,e.\ with $f = f_\alpha$.
\end{fact}

This article investigates paired functions as defined below.

\begin{definition}\label{def:paired}
	Let $I,J \subseteq V$ be disjoint sets.
	We define the \emph{paired} function
	\[
		p^{(V,k)}_{I,J} = p_{I,J} = f_{I,J^\complement} + f_{J,I^\complement}\text{.}
	\]
	The functions $f_{I,J^\complement}$ and $f_{J,I^\complement}$ will be called the \emph{legs} of $p_{I,J}$.%
	\footnote{
		We note that the term \emph{leg} depends on the representation of a paired function as $p_{I,J}$, which may not be unique.
		In many cases, however, the set $\{I,J\}$ is indeed uniquely determined, as we will see in Theorem~\ref{thm:paired_determines_I_J}.
	}
	Moreover, let%
	\footnote{Warning: The border case $p_{\emptyset,\emptyset}$ is non-Boolean and hence $p_{\emptyset,\emptyset} \neq \chi_{\mathcal{P}_{\emptyset,\emptyset}}$, see the discussion in Section~\ref{sec:elem_paired}.}
	\[
	\mathcal{P}_{I,J}^{(V,k)}
	= \mathcal{P}_{I,J}
	= \supp(p_{I,J})
	= \mathcal{F}_{\!I,J^\complement} \cup \mathcal{F}_{\!J,I^\complement}\text{.}
	\]
\end{definition}

The above results immediately give the subsequent bound on the degree of paired functions.

\begin{lemma}[Elementary bound]\label{lem:elem_bound}
	Let $I,J\subseteq V$ be disjoint of size $i = \#I$ and $j = \#J$.
	Then
	\[
		\deg(p_{I,J}) \leq \min(i+j,k,n-k)\text{.}
	\]
\end{lemma}

\begin{proof}
By Fact~\ref{fct:basic_deg}, both legs are of degree at most $\min(i+j,k,n-k)$.
Now the statement follows from Fact~\ref{fct:deg_linear_combination}.
\end{proof}

This article is motivated by the fact that -- as it turns out -- the elementary bound it is not always sharp.
Our main result is the following determination of the exact degree of a paired function.

\begin{theorem}\label{thm:paired_deg}
	Let $I,J\subseteq V$ be disjoint of size $i = \#I$ and $j = \#J$.
	Then
	\[
		\deg p^{(V,k)}_{I,J}
		= \begin{cases}
			i + j - 1 & \text{if }i+j\text{ odd and }i+j \leq \min(k,n-k)\text{,} \\
			k - 1 & \text{if }k\text{ odd and }n=2k\text{ and }i+j \geq k\text{,} \\
			\min(i+j,k,n-k) & \text{otherwise.}
		\end{cases}
	\]
\end{theorem}

The structure of the article is as follows:
Section~\ref{sec:prelim} introduces the necessary preliminaries.
Sections~\ref{sec:elem_basic} and~\ref{sec:elem_paired} establish elementary properties of basic and paired functions.
The latter includes Theorem~\ref{thm:paired_determines_I_J}, which shows that a paired function determines its legs essentially uniquely, except for a few boundary cases.
Building on these foundations, Section~\ref{sec:main_proof} proves the above stated Theorem~\ref{thm:paired_deg} by first considering the case $n = 2k$, and then reducing the general case to this special instance.
Section~\ref{sec:hartman} presents an application to Hartman's conjecture in design theory.
The final Section~\ref{sec:small_size} explores fixed-degree Boolean functions of small size.
Theorem~\ref{thm:paired_vs_pencil} characterizes all paired functions whose size falls below or matches the pencil bound.
We conclude with a computational investigation of the minimal sizes of fixed-degree Boolean functions, whose results are summarized in Table~\ref{tbl:t2},~\ref{tbl:t3}, and~\ref{tbl:t4}.
This leads to Conjecture~\ref{conj:pencil_paired_are_smallest}, which posits that functions of minimal size can always be found among the pencils, dual pencils, or, in certain cases, paired functions.

\section{Preliminaries}\label{sec:prelim}

\subsection{The degree, designs, and antidesigns}

We collect a few further notions and results about functions on the slice.
Let $f : \binom{V}{k} \to \R$.
We define three kinds of elementary modifications of $f$.
\begin{itemize}
	\item The \emph{dual} (or \emph{complementary}) function
	\[
		f^\perp : \binom{V}{n-k}\to\R\text{,}\quad B \mapsto f(B^\complement)\text{.}
	\]
	\item For $k \geq 1$, the \emph{derived} function in $x\in V$
	\[
		\Der_x(f) : \binom{V\setminus\{x\}}{k-1} \to \R\text{,}\quad K \mapsto f(K \cup \{x\})\text{.}
	\]
	\item For $n-k \geq 1$, the \emph{residual} function in $x\in V$
	\[
		\Res_x(f) : \binom{V\setminus\{x\}}{k} \to \R\text{,}\quad K \mapsto f(K)\text{.}
	\]
\end{itemize}
Clearly,
\[
	\#f^\perp = \#f\text{.}
\]
The functions $\Der_x(f)$ and $\Res_x(f)$ essentially split the domain $\binom{V}{k}$ of $f$ into those blocks containing $x$, and not containing $x$, respectively.
As a consequence,
\[
	\#\Der_x(f) + \#\Res_x(f) = \#f\text{.}
\]

Concerning the degrees of the above modifications, the following statements are known.

\begin{fact}[{{\cite[Th.~5.3]{Kiermaier-Mannaert-Wassermann-2025-JCTSA212:P105979}}}]\label{fct:dual_deg}
	Let $f : \binom{V}{k} \to \R$.
	Then $\deg(f^\perp) = \deg(f)$.
\end{fact}

\begin{fact}[{{\cite[Cor.~6.11]{Kiermaier-Mannaert-Wassermann-2025-JCTSA212:P105979}}}]\label{fct:der_res}
	Let $f : \binom{V}{k} \to \R$ and $x\in V$.
	\begin{enumerate}[(a)]
		\item If $k \geq 1$, then $\deg_{V\setminus\{x\}}(\Der_x(f)) \leq \deg_V(f)$.
		\item If $n-k \geq 1$, then $\deg_{V\setminus\{x\}}(\Res_x(f)) \leq \deg_V(f)$.
	\end{enumerate}
\end{fact}

As already mentioned in Section~\ref{sect:intro}, a function $f : \binom{V}{k} \to \R$ is of degree $0$ if and only if $f$ is non-zero and constant.
The Boolean functions of degree~$1$ are known, too:
\begin{fact}[{{\cite{Meyerowitz-1992-JCombinInformSystemSci17[1-2]:39-42} and \cite[Th.~1.2]{Filmus-Ihringer-2019-JCTSA162:241-270}}}]\label{fct:classification_degree_1}
	Every Boolean function $\binom{V}{k} \to \{0,1\}$ of degree $1$ is basic.
	So for $\min(k,n-k) \neq 0$, these functions are precisely the pencils $f_{\{x\},V}$ and the dual pencils $f_{\emptyset,V\setminus\{x\}} = f_{\{x\},V}^\perp$ with $x\in V$.
\end{fact}

A set $D \subseteq \binom{V}{k}$ is called a $t$-$(v,k,\lambda)$ \emph{design} if for all $T\in\binom{V}{t}$, there are exactly $\lambda$ elements of $D$ (called \emph{blocks}) containing $T$.
For any $v$, $k$ and $t$ there are always the \emph{trivial} designs, namely the \emph{empty} $t$-$(v,k,0)$ design $\emptyset$ and the \emph{complete} $t$-$(v,k,\lambda_{\max})$ design $\binom{V}{k}$ with $\lambda_{\max} = \binom{v-t}{k-t}$.
Clearly, for any design, we have $0 \leq \lambda \leq \lambda_{\max}$.
The number $\lambda_{\max} / \lambda$ is known as the \emph{index}, and designs of index $2$ are called \emph{halvings}.

For all $i\in\{0,\ldots,t\}$, any $t$-$(v,k,\lambda)$ design $D$ is also an $i$-$(v,k,\lambda_i)$ design where
\begin{equation}\label{eq:lambda_i}
	\lambda_i
	= \lambda \cdot\frac{\binom{v-i}{t-i}}{\binom{k-i}{t-i}}
	= \lambda \cdot\frac{\binom{v-i}{k-i}}{\binom{v-t}{k-t}}\text{.}
\end{equation}
In particular,
\[
	\#D = \lambda_0 = \lambda \cdot \frac{\binom{v}{t}}{\binom{k}{t}} = \lambda \cdot \frac{\binom{v}{k}}{\binom{v-t}{k-t}}\text{.}
\]
The equations~\eqref{eq:lambda_i} imply the \emph{integrality conditions} on a numerical parameter set $t$-$(v,k,\lambda)$:
If the parameters are \emph{realizable} (i.\,e.\ a design with these parameters does exist), then they are \emph{admissible} in the sense that all numbers $\lambda_0,\ldots,\lambda_t$ are integers.

For $S\subseteq V$ and $i\in\N$, the intersection number
\[
	\alpha_i(S) = \#\{B\in D \mid \#(S \cap B) = i\}
\]
is defined.
Intersection numbers obey the \emph{Mendelsohn equations} \cite[Th.~1]{Mendelsohn-1971}, \cite[Satz~2]{Oberschelp-1972-MPSemBNF19:55-67}
\[
	\sum_{j=0}^k \binom{j}{i} \alpha_j(S) = \binom{s}{i} \lambda_i \quad (i\in\{0,\ldots,t\})\text{.}
\]
Via Gauss elimination, these system of linear equations is transformed into the \emph{Köhler equations}~\cite[Satz 1]{Koehler-1989-DM73[1-2]:133-142}, see also \cite[Satz~1]{Harnau-1988-RostMKoll34:47-52}, \cite[Th.~2.6]{Kiermaier-Pavcevic-2015-JCD23[11]:463-480}.

\begin{fact}[Köhler equations]\label{fct:koehler}
	Let $D$ be a $t$-$(v,k,\lambda)$ design, $S \subseteq V$ and $s = \#S$.
	For $i\in\{0,\ldots,t\}$, a parametrization of the intersection number $\alpha_i(S)$ by $\alpha_{t+1}(S),\ldots,\alpha_k(S)$ is given by
	\[
		\alpha_i(S) = \binom{s}{i} \sum_{j=i}^t (-1)^{j-i} \binom{s-i}{j-i} \lambda_i + (-1)^{t+1-i} \sum_{j=t+1}^k \binom{j}{i} \binom{j-i-1}{t-i} \alpha_i(S)\text{.}
	\]
\end{fact}

There is the following connection between $t$-designs and functions of degree $t$, see for example the discussion in \cite[Sec.~4]{Kiermaier-Mannaert-Wassermann-2025-JCTSA212:P105979}.
\begin{fact}[Design orthogonality]\label{fct:design_orthogonality}
	Let $A,D \subseteq \binom{V}{k}$ such that $\deg(A) \leq t$ and $D$ is a $t$-design.
	Then
	\[
		\#(A \cap D) = \frac{\lambda}{\lambda_{\max}}\cdot\#A
	\]
\end{fact}

The following property is a further consequence of the connection to designs and their integrality conditions.

\begin{fact}[{{Divisibility property, \cite[Th.~4.7]{Kiermaier-Mannaert-Wassermann-2025-JCTSA212:P105979}}}]\label{fct:divisibility}
	Let $A \subseteq \binom{V}{k}$ be a non-empty set of degree $t$.
	Then
	\[
		\gcd\left(\tbinom{n-0}{k-0}, \tbinom{n-1}{k-1}, \ldots, \tbinom{n-t}{k-t}\right) \mid \#A\text{.}
	\]
\end{fact}

\subsection{Association schemes}
We give a brief outline of the relevant parts of the theory of association schemes by Delsarte \cite{Delsarte-1973-PhilRRSuppl10}.
Let $X \neq \emptyset$ be a finite set.
A partition $\mathcal{R} = \{R_0,\ldots,R_d\}$ of $X\times X$ is called a \emph{(symmetric) association scheme} if $R_0$ is the identity relation, all relations $R_i$ are symmetric, and there exist \emph{intersection numbers} $p_{ij}^\ell$ such that for all $(x,y)\in R_\ell$, the number of $z\in X$ with $(x,z)\in R_i$ and $(z,y)\in R_j$ equals $p_{ij}^\ell$.
By the properties of an association scheme, the matrix space $\mathcal{B} = \langle D_0,\ldots,D_d\rangle_{\R}$, spanned by the adjacency matrices $D_i$ of the relations $R_i$, is a commutative unital $\R$-algebra consisting of symmetric matrices, known as the \emph{Bose-Mesner-Algebra} of $\mathcal{R}$.
Its elements are simultaneously diagonalizable by an orthogonal matrix, and the primitive idempotents of $\mathcal{B}$ provide a second $\R$-basis $\{E_0,\ldots,E_d\}$ of $\mathcal{B}$ with $E_0 = \frac{1}{\#X} J$, where $J$ denotes the all-one matrix.
Hence, there exist unique real numbers $P_i(\ell),Q_\ell(i)\in\R$ with $D_i = P_i(0) E_0 + \ldots + P_i(d) E_d$ and $E_\ell = Q_{\ell}(0) D_0 + \ldots + Q_{\ell}(d) E_d$.
The matrices $P = (P_{i}(\ell))_{i\ell}$ and $Q = (Q_{\ell}(i))_{\ell i}$ are known as the \emph{first} and the \emph{second eigenmatrix} of $\mathcal{R}$, respectively.
The association scheme $\mathcal{R}$ is called \emph{$Q$-polynomial} with respect to a fixed order of the matrices $E_{\ell}$ if there exist polynomials $f_0,\ldots,f_d\in \R[X]$ of degree $\deg(f_\ell) = \ell$ and positions $z_0,\ldots,z_d\in\R$ such that $Q_{\ell i} = f_\ell(z_i)$.

For a non-empty set $Y \subseteq X$, the \emph{inner distribution} $(a_0,\ldots,a_d)$ of $Y$ is defined by $a_i = \frac{\#((Y \times Y) \cap R_i)}{\#Y}$, and the \emph{dual inner distribution} as $(b_0,\ldots,b_d)^\top = Q(a_0,\ldots,a_d)^\top$.
Clearly, $a_0 = 1$ and $a_0 + \ldots + a_d = \#Y$.
An important property of the dual distribution is $b_\ell \geq 0$ for all $\ell\in\{0,\ldots,d\}$, which is the basis for Delsarte's LP-bound.
In a $Q$-polynomial association scheme, $Y$ is called a \emph{$t$-design} if $b_1 = \ldots = b_t = 0$, and a \emph{$t$-antidesign} if $b_{t+1} = \ldots = b_d = 0$.%
\footnote{
	Note that this definition depends on the chosen order of the matrices $E_\ell$.
}

For $X = \binom{V}{k}$, the \emph{Johnson scheme} $J(n,k)$ is defined by the relations $R_i = \{(A,B)\in X\times X \mid \#(A \cap B) = k-i\}$ where $i\in\{0,\ldots,k\}$.
The eigenmatrices $P$ and $Q$ are given by evaluations of so-called Eberlein- and Hahn-Polynomials.
The Johnson scheme is $Q$-polynomial.
With respect to the corresponding order of the matrices $E_\ell$,
\[
	Q_\ell(i) = \left(\binom{n}{\ell} - \binom{n}{\ell-1}\right)\sum_{j=0}^\ell (-1)^j \frac{\binom{\ell}{j}\binom{n+1-\ell}{j}}{\binom{k}{j}\binom{n-k}{j}} \binom{i}{j}\text{.}
\]
A non-empty set $Y \subseteq \binom{V}{k}$ is a $t$-design in the Johnson scheme if and only if it is a combinatorial $t$-design.
Moreover, we have the following alternative characterization of the degree:
\begin{fact}
	Let $Y \subseteq \binom{V}{k}$ be non-empty.
	Then $\deg(Y) \geq t$ if and only if $Y$ is a $t$-antidesign in the Johnson scheme $J(n,k)$.
\end{fact}

We record the following consequence.

\begin{lemma}[LP-bound for the degree]\label{lem:lp_bound}
	Let $Y \subseteq \binom{V}{k}$ be a non-empty set of degree $t$.
	Then $\#Y$ is lower bounded by the optimal value of following linear program in the variables $a_0,\ldots,a_d\in\R$:
	\[
		\begin{tblr}{lll}
			\text{minimize} & a_0 + \ldots + a_d \\
			\text{subject to} & Q_\ell(0) a_0 + \ldots + Q_\ell(d) a_d = 0 & \text{for all }\ell\in\{t+1,\ldots,d\}\text{,} \\
			\text{and} & Q_\ell(0) a_0 + \ldots + Q_\ell(d) a_d \geq 0 & \text{for all }\ell\in\{0,\ldots,t\}\text{,} \\
			\text{and} & a_i \geq 0 & \text{for all }i\in\{0,\ldots,d\}\text{.} \\
		\end{tblr}
	\]
\end{lemma}

\section{Elementary properties of basic functions}\label{sec:elem_basic}

In this section, we consider sets $I \subseteq J \subseteq V$ of size $i = \#I$ and $j = \#J$.
Clearly, all basic functions $f_{I,J}$ are Boolean, and their size is
\[
	\#f_{I,J} = \binom{j-i}{k-i}\text{.}
\]
The basic functions include the special case $f_{\emptyset,V} = \vek{1}$.

\begin{lemma}\label{lem:basic}
	Let $I,J$ be subsets of $V$ with $I \subseteq J$ of sizes $i = \#I$ and $j = \#J$.
	\begin{enumerate}[(a)]
		\item\label{lem:basic:0} $f_{I,J} = \vek{0}$ if and only if $i > k$ or $j < k$.
		\item\label{lem:basic:size1} $\#f_{I,J} = 1$ if and only if $i = k$ or $j = k$.
		More precisely, for $K\in\binom{V}{k}$ and $I' \subseteq J' \subseteq V$, we have $\mathcal{F}_{\!I',J'} = \{K\}$ if and only if $I=K$ or $J=K$.
		\item\label{lem:basic:generic} $\#f_{I,J} \geq 2$ if and only if $i < k < j$.
		In this case, the sets $I$ and $J$ are uniquely determined by $f_{I,J}$ as $I = \bigcap \mathcal{F}_{\!I,J}$ and $J = \bigcup \mathcal{F}_{\!I,J}$.
		\item\label{lem:basic:1} $f_{I,J} = \vek{1}$ if and only if $I = \emptyset$ and $J = V$.
	\end{enumerate}
\end{lemma}

\begin{proof}
	Most statements are straightforward to check.
	We only show the second statement of Part~\ref{lem:basic:generic} for $I$, the one for $J$ is then done analogously.
	By $\#f_{I,J} \geq 2$, there exist two different blocks $K_1, K_2\in\mathcal{F}_{\!I,J}$.
	It is clear that $I \subseteq \bar{I} \coloneqq \bigcap \mathcal{F}_{\!I,J}$.
	Assuming $I \neq \bar{I}$, there exists an $x\in\bar{I} \setminus I$.
	Moreover, there exists a $y\in K_1\setminus K_2$.
	As $\bar{I}$ is a subset of both $K_1$ and $K_2$, $y \notin \bar{I}$.
	Now $K' = (K_2 \setminus \{x\}) \cup\{y\}$ is an element of $\mathcal{F}_{\!I,J}$, but $\bar{I} \nsubseteq K'$, which is a contradiction.
\end{proof}

\newpage

\begin{lemma}[Pascal decomposition]
	\label{lem:basic_pascal_split}
	Let $I \subseteq J\subseteq V$ and $x\in J \setminus I$.%
	\footnote{
		In the case $I = J$, there is no suitable $x$.
	}
	Then $\mathcal{F}_{\!I,J}$ is the disjoint union of $\mathcal{F}_{\!I\cup\{x\},J}$ and $\mathcal{F}_{\!I,J\setminus\{x\}}$.
\end{lemma}

\begin{proof}
	The elements of $\mathcal{F}_{\!I,J}$ are partitioned into two parts, depending on whether they contain the element $x$ or not.
\end{proof}

\begin{remark}
	The bijective standard proof of Pascal's identity for binomial coefficients is based on the decomposition idea in Lemma~\ref{lem:basic_pascal_split}, which is the reason why we called it \emph{Pascal decomposition}.
	Indeed, for the cardinalities, the decomposition implies the identity
	\[
		\binom{j-i}{k-i} = \binom{j-i-1}{k-i-1} + \binom{j-i-1}{k-i}
	\]
	where $i = \#I$ and $j = \#J$.
\end{remark}

\begin{lemma}\label{lem:basic_elem_mod}
	Let $I \subseteq J\subseteq V$.
	\begin{enumerate}[(a)]
		\item $(f^{(V,k)}_{I,J})^\perp = f^{(V,n-k)}_{J^\complement, I^\complement}$.
		\item For $x\in V$ and $k \geq 1$,
		\[
			\Der_x(f^{(V,k)}_{I,J})
			= \begin{cases}
				f^{(V\setminus\{x\},k-1)}_{I\setminus\{x\},J\setminus\{x\}} & \text{if }x\in I\text{,} \\[2mm]
				\vek{0} & \text{if }x\in V\setminus J\text{,} \\[2mm]
				f^{(V\setminus\{x\},k-1)}_{I,J\setminus\{x\}} & \text{if }x \in J \setminus I\text{.}
			\end{cases}
		\]
		\item For $x\in V$ and $n-k \geq 1$,
		\[
			\Res_x(f^{(V,k)}_{I,J})
			= \begin{cases}
				\vek{0} & \text{if }x\in I\text{,} \\[2mm]
				f^{(V\setminus\{x\},k)}_{I,J} & \text{if }x\in V \setminus J\text{,} \\[2mm]
				f^{(V\setminus\{x\},k)}_{I,J\setminus\{x\}} & \text{if }x \in J\setminus I\text{.}
			\end{cases}
		\]
	\end{enumerate}
\end{lemma}

\begin{proof}
	Straightforward.
\end{proof}

\newpage

\section{Elementary properties of paired functions}\label{sec:elem_paired}
We start with two obvious but important properties.

\begin{lemma}[Symmetry]\label{lem:p_symm}
For all disjoint set $I,J\subseteq V$,
\[
	p_{I,J} = p_{J,I}\text{.}
\]
\end{lemma}

\begin{lemma}\label{lem:p_size}
The size of a paired function is
\[
	\#p_{I,J}
	= \#f_{I,J^\complement} + \#f_{J,I^\complement}
	= \binom{n-i-j}{k-i} + \binom{n-i-j}{k-j}\text{.}
\]
\end{lemma}

\begin{lemma}[Pascal decomposition]
	\label{lem:paired_pascal_split}
	Let $I,J\subseteq V$ be disjoint and $x\in V\setminus (I \cup J)$.%
	\footnote{
		In the case $I \cup J = V$, there is no suitable $x$.
	}
	Then
	\[
		p_{I,J} = p_{I\cup \{x\}, J} + p_{I, J\cup\{x\}}\text{.}
	\]
\end{lemma}

\begin{proof}
	By Lemma~\ref{lem:basic_pascal_split},
	\begin{multline*}
		p_{I,J}
		= f_{I,J^\complement} + f_{J,I^\complement} \\
		= f_{I \cup\{x\}, J^\complement} + f_{I, (J \cup \{x\})^\complement}
		+ f_{J \cup \{x\}, I^\complement} + f_{J, (I\cup \{x\})^\complement}
		= p_{I \cup \{x\}, J} + p_{I,J\cup \{x\}}\text{.}
	\end{multline*}
\end{proof}

With the exception of
\begin{equation}\label{eq:paired_i=j=0}
	p_{\emptyset, \emptyset} = 2\cdot \vek{1}\text{,}
\end{equation}
all paired functions $p_{I,J}$ are Boolean.
Because, if two disjoint sets $I,J\subseteq V$ are not both empty, say (up to symmetry) $x\in I$, then this $x$ is contained in all blocks of $f_{I,J^\complement}$ and in none of $f_{J,I^\complement}$.
Hence, unless $I = J = \emptyset$, the two legs $\mathcal{F}_{\!I,J^\complement} = \supp(f_{I,J^\complement})$ and $\mathcal{F}_{\!J,I^\complement} = \supp(f_{J,I^\complement})$ of $p_{I,J}$ are disjoint.

\begin{example}\label{ex:small_paired_sets}
Small non-trivial examples of paired sets are
\begin{align*}
	\mathcal{P}^{(\{1,\ldots,6\},3)}_{\emptyset,\{123\}} & = \big\{\{123\},\,\,\{456\}\big\}\text{,}\\
	\mathcal{P}^{(\{1,\ldots,7\},3)}_{\emptyset,\{123\}} & = \big\{\{123\},\{124\},\{134\},\{234\},\,\,\{567\}\big\}\text{,}\\
	\mathcal{P}^{(\{1,\ldots,7\},3)}_{\{1\},\{67\}} & = \big\{\{123\},\{124\},\{125\},\{134\},\{135\},\{145\},\,\,\{267\},\{367\},\{467\},\{567\}\big\}\text{.}
\end{align*}
In the listings, the two legs were visually set slightly apart.

According to Theorem~\ref{thm:paired_deg}, which is yet to be proven, all these sets are of degree $2$.
\end{example}

Let $i = \#I$ and $j = \#J$.
For $\min(i,j) = 0$, we note that
\begin{equation}\label{eq:paired_i=0}
	\mathcal{P}^{(V,k)}_{I,\emptyset}
	= \mathcal{P}^{(V,k)}_{\emptyset, I}
	= \binom{I^\complement}{k} \cup \binom{I^\complement}{n-k}^{\!\complement}\text{.}
\end{equation}
In the subcase $\{i,j\} = \{0,1\}$, which is equivalent to $i + j = 1$, we have $I \cup J = \{x\}$, and the legs of $\mathcal{P}_{I,J}$ are the $k$-subsets of $V$ containing $x$ and the $k$-subsets not containing $x$, respectively.
Hence
\begin{equation}\label{eq:paired_i+j=1}
	p_{\{x\},\emptyset} = p_{\emptyset,\{x\}} = \vek{1} = f_{\emptyset,\emptyset^\complement}\text{.}
\end{equation}

\begin{lemma}\label{lem:p_legs_basic}
	Let $I,J \subseteq V$ be disjoint sets of size $i = \#I$ and $j = \#J$.
	\begin{enumerate}[(a)]
		\item\label{lem:p_legs_basic:first} The first leg $f_{I,J^\complement}$ of $p_{I,J}$ is empty $\iff p_{I,J} = f_{J,I^\complement} \iff k < i \text{ or } j > n-k$.
		\item\label{lem:p_legs_basic:second} The second leg $f_{J,I^\complement}$ of $p_{I,J}$ is empty $\iff p_{I,J} = f_{I,J^\complement} \iff k < j \text{ or } i > n-k$.
	\end{enumerate}
\end{lemma}

\begin{proof}
	This follows from Lemma~\ref{lem:basic}\ref{lem:basic:0}.
\end{proof}

\begin{lemma}\label{lem:p_legs}
	Let $I,J\subseteq V$ be disjoint of size $i = \#I$ and $j = \#J$.
	\begin{enumerate}[(a)]
	\item\label{lem:p_legs:empty} $p_{I,J} = \vek{0} \iff \min(i,j) > \min(k,n-k)$ or $\max(i,j) > \max(k,n-k)$.
	\item\label{lem:p_legs:one_empty} $p_{I,J}$ has precisely one non-empty leg\\
	\null\hfill$\iff  \min(i,j) \leq \min(k,n-k) < \max(i,j) \leq \max(k,n-k)$.
	\item\label{lem:p_legs:both_empty} Both legs of $p_{I,J}$ are non-empty $\iff \max(i,j) \leq \min(k,n-k)$.
	\end{enumerate}
\end{lemma}

\begin{proof}
This follows from a combination of both statements of Lemma~\ref{lem:p_legs_basic}.
\end{proof}

\begin{lemma}\label{lem:paired_size_1}
Let $I,J\subseteq V$ be disjoint of size $i = \#I$ and $j = \#J$.
Then
\[
	\#p_{I,J} = 1 \iff \begin{cases} \min(i,j) = \min(k,n-k) < \max(i,j) \leq \max(k,n-k)\quad \text{or}\\\min(i,j) \leq \min(k,n-k) < \max(i,j) = \max(k,n-k)\text{.}\end{cases}
\]
Note that the middle \enquote{$<$} in the second condition is equivalent to $n \neq 2k$.
\end{lemma}

\begin{proof}
	The property $\#p_{I,J} = 1$ is equivalent to one leg being empty and the other leg having size $1$.
	Now use Lemma~\ref{lem:p_legs}\ref{lem:p_legs:one_empty} and~\ref{lem:basic}\ref{lem:basic:size1}.
\end{proof}

\begin{lemma}\label{lem:paired_size_2}
Let $I,J\subseteq V$ be disjoint of size $i = \#I$ and $j = \#J$.
Then
\[
	\#p_{I,J} = 2
	\iff \begin{cases}
		i = j = \min(k,n-k)\quad\text{or} \\
		n=2k\text{ and }\max(i,j) = k\text{ and }i\neq j\quad\text{or} \\
		n\notin\{2k-1,\,2k,\,2k+1\} \text{ and } \{i,j\} = \{k-1,\, n-k-1\}\text{.}
	\end{cases}
\]
The second case can be rewritten as $\min(i,j) < \min(k,n-k) = \max(k,n-k) = \max(i,j)$ and the third as $\min(i,j) + 1 = \min(k,n-k) < \max(k,n-k)-1 = \max(i,j)$.
\end{lemma}

\begin{proof}
	The property $\#p_{I,J} = 2$ is equivalent to either having two legs of size 1 (resulting in the first and second case) or one empty leg and one of size 2 (the third case).
\end{proof}

\begin{lemma}\label{lem:i=j=k}
	Let $K_1, K_2$ be disjoint subsets of $V$ of size $k$ and let $g = \chi_{K_1} + \chi_{K_2}$.
	\begin{enumerate}[(a)]
	\item\label{lem:i=j=k:double} $g = 2\cdot \vek{1}$ if and only if $k=0$.
	\item\label{lem:i=j=k:full} $g = \vek{1}$ if and only if $n=2$ and $k=1$.
	\item\label{lem:i=j=k:basic} $g$ is a basic function if and only if $k=1$.
	In this case, $g = f_{\emptyset, K_1 \cup K_2}$ is the unique representation of $g$ as a basic function.
	\item\label{lem:i=j=k:paired} $g$ is always a paired function.
	The pairs $\{I,J\}$ ($I,J\subseteq V$ disjoint) with $g = p_{I,J}$ are precisely the following:
	\begin{itemize}
		\item $\{K_1,K_2\}$,
		\item for $k=1$ and $n \geq 4$, additionally $\{\emptyset, (K_1 \cup K_2)^\complement\}$,
		\item for $n = 2k$, additionally all $\{K_1,J\}$ and $\{I,K_2\}$ with $I \subseteq K_1$ and $J\subseteq K_2$.
	\end{itemize}
	\end{enumerate}
\end{lemma}

\begin{proof}
	For $k=0$, the statements are easy to verify, so we may assume that $g$ is Boolean.

	Part~\ref{lem:i=j=k:double} and~\ref{lem:i=j=k:full} are straightforward to check.

	Part~\ref{lem:i=j=k:basic}:
	For \enquote{$\Rightarrow$}, let $g = f_{I,J}$ be basic with disjoint sets $I,J\subseteq V$.
	By Lemma~\ref{lem:basic}\ref{lem:basic:generic} and $\#g = 2$, we obtain $I = \bigcap \supp(g) = K_1 \cap K_2 = \emptyset$.
	Hence, $2 = \#g = \#f_{\emptyset,J} = \binom{\#J}{k}$ and therefore $\#J = 2$ and $k=1$.
	The second leg $f_{J^\complement, V}$ must be empty, such that Lemma~\ref{lem:basic}\ref{lem:basic:0} implies $n-2 = \#J^\complement > k = 1$, i.\,e. $n > 3$.
	For \enquote{$\Leftarrow$}, let $k = 1$ and $n \geq 4$.
	Then, clearly, $g = f_{\emptyset, K_1 \cup K_2}$, and by Lemma~\ref{lem:basic}\ref{lem:basic:generic}, this representation as a basic function is unique.

	Part~\ref{lem:i=j=k:paired}:
	The direction \enquote{$\Leftarrow$} is routinely checked.
	For \enquote{$\Rightarrow$}, let $g = p_{I,J}$ be a paired function with $I,J\subseteq V$ disjoint.
	If a leg of $g$ is empty, $g$ is a basic function.
	Part~\ref{lem:i=j=k:basic} implies $k=1$, and the unique representation yields the pair $\{I,J\} = \{\emptyset, (K_1 \cup K_2)^\complement\}$.
	If both legs of $g$ are non-empty, they are basic functions of size $1$.
	The application of Lemma~\ref{lem:basic}\ref{lem:basic:size1} to both legs implies, up to swapping $I$ and $J$, that
	\[
		(I = K_1 \text{ or }J^\complement = K_1)\; \text{ and }\;(J = K_2 \text{ or }I^\complement = K_2)\text{.}
	\]
	Hence, we are in one of these four cases:
	\begin{align*}
		I = K_1 &\;\text{ and }\; J = K_2\text{,} &
		I = K_1 &\;\text{ and }\; I = K_2^\complement\text{,} \\
		J = K_1^\complement &\;\text{ and }\; J = K_2\text{,} &
		J = K_1^\complement &\;\text{ and }\; I = K_2^\complement\text{.}
	\end{align*}
	The first case is $\{I,J\} = \{K_1,K_2\}$.
	The remaining three cases are only possible for $n = 2k$, implying that $K_2 = K_1^\complement$, where for the fourth case we used that $I$ and $J$ are disjoint.
	The second case gives $\{I,J\} = \{K_1,J\}$ with $I$ and $J$ disjoint, i.\ e. $J \subseteq K_2$.
	The third case yields $\{I,J\} = \{I,K_2\}$ with $I \subseteq K_1$, and the fourth case once again $\{I,J\} = \{K_1,K_2\}$.
\end{proof}

\begin{lemma}\label{lem:paired:meet_join}
	Let $I,J\subseteq V$ be disjoint of size $i = \#I$ and $j = \#J$ such that both legs of $p_{I,J}$ are non-empty, i.\,e. such that $\max(i,j) \leq \min(k,n-k)$.
	Then
	\[
		\bigcap \mathcal{P}_{I,J}
		= \begin{cases}
			(I\cup J)^\complement & \text{if }i=j=n-k\text{,} \\
			\emptyset & \text{otherwise}
		\end{cases}
		\qquad\text{and}\qquad
		\bigcup \mathcal{P}_{I,J}
		= \begin{cases}
			I\cup J & \text{if }i=j=k\text{,} \\
			V & \text{otherwise.}
		\end{cases}
	\]
	For $n = 2k$, we get that, in particular, $\bigcap \mathcal{P}_{I,J} = \emptyset$ and $\bigcup \mathcal{P}_{I,J} = V$.
\end{lemma}

\begin{proof}
	For $i=j=k$, $p_{I,J} = \{I,J\}$ and for $i=j=n-k$, $p_{I,J} = \{I^\complement, J^\complement\}$.
	This verifies the stated special cases.

	In the case that $i=j=k$ is false, up to swapping $I$ and $J$, we may assume $i < k$.
	We show that all $x\in V$ are contained in some block of $\mathcal{P}_{I,J}$, implying that $\bigcup\mathcal{P}_{I,J} = V$.
	For $x\in I \cup J$, this follows from the assumption that both legs are non-empty.
	Furthermore, by $i < k$, each $x\in V\setminus (I \cup J)$ is contained in some block of the leg $\mathcal{F}_{\!I,J^\complement}$.

	In the case that $i=j=n-k$ is false, $\bigcap\mathcal{P}_{I,J} = \emptyset$ is shown similarly.
\end{proof}

\begin{lemma}\label{lem:paired:full}
	$p_{I,J} = \vek{1}$ precisely in the following cases.
	\begin{enumerate}[(i)]
		\item\label{lem:paired:full:i+j=1} $i + j = 1$;
		\item\label{lem:paired:full:k=i=j=1} $n=2$ and $k=i=j=1$;
		\item\label{lem:paired:full:k=0} $k \in\{0,n\}$ and $\min(i,j) = 0$ and $\max(i,j) \neq 0$.
	\end{enumerate}
\end{lemma}

\begin{proof}
The direction \enquote{$\Leftarrow$} is readily checked.
For \enquote{$\Rightarrow$}, let $p_{I,J} = \vek{1}$.
By Equation~\eqref{eq:paired_i=j=0}, $\max(i,j) \geq 1$.
From Lemma~\ref{lem:p_legs}\ref{lem:p_legs:empty}, $\min(i,j) \leq \min(k,n-k)$ and $\max(i,j) \leq \max(k,n-k)$.
For $k\in\{0,n\}$, we get $\min(i,j) = 0$.
Together with Equation~\eqref{eq:paired_i=j=0} and~\eqref{eq:paired_i+j=1}, it remains to consider $1 \leq k \leq n-1$ and $i + j \geq 2$.
We distinguish the following cases.
\begin{itemize}
	\item For $k = 1$ and $\max(i,j) \geq 2$, by symmetry, we may assume $i \geq 2$.
	Now $\mathcal{F}_{\!I,J^\complement} = \emptyset$ and hence $\mathcal{P}_{I,J} = \mathcal{F}_{\!J, I^\complement} \subseteq \binom{I^\complement}{1}$, which by $i \geq 2$ is not $\binom{V}{1}$.
	\item For $k = 1$ and $\max(i,j) \leq 1$, we have $i = j = 1$ and one obtains $\mathcal{P}_{I,J} = \binom{I\cup J}{1}$.
	For $n = 2$, this set equals $\binom{V}{1}$.
	For $n \geq 2$, there exists an $x\in V\setminus\{I\cup J\}$, resulting in $\{x\} \notin \mathcal{P}_{I,J}$.
	\item For $k \geq 2$ and $\min(i,j) = 0$, we may assume $j = 0$.
	Then $2 \leq i$ and $i = \max(i,j) \leq \max(k,n-k) \leq n-1$.
	So there exist elements $x\in I$ and $y \in I^\complement$.
	Now any $k$-subset of $V$ containing both $x$ and $y$ (which exists since $k \geq 2$) will not be contained in $\mathcal{P}_{I,\emptyset} = \binom{I}{k} \cup \binom{I^\complement}{k}$ (see~\ref{eq:paired_i=0}).
	\item For $k \geq 2$ and $\min(i,j) \geq 1$, there exist $x\in I$ and $y\in J$, and any $k$-subset of $V$ containing both $x$ and $y$ (which exists by $k \geq 2$) is not contained in $\mathcal{P}_{I,J}$.
\end{itemize}
Hence, we are in one of the cases listed in the statement.
\end{proof}

\begin{lemma}\label{lem:paired_is_basic}
	Let $I,J \subseteq V$ be disjoint, and denote their sizes by $i = \#I$, and $j = \#J$.
	Then the paired function $p_{I,J}$ equals a basic function $f_{I',J'}$ with disjoint $I',J' \subseteq V$ precisely in the following cases:
	\begin{enumerate}[(i)]
		\item\label{lem:paired_is_basic:one_leg_empty} $\max(i,j) > \min(k,n-k)$, where at least one of the legs is empty;
		\item\label{lem:paired_is_basic:k=1} $\min(k,n-k) = i = j = 1$, where $p_{I,J} = f_{\emptyset, I\cup J}$ for $k =1$ and $p_{I,J} = f_{(I\cup J)^\complement, V}$ for $k = n-1$;
		\item\label{lem:paired_is_basic:i+j=1} $i+j = 1$, where $p_{I,J} = \vek{1} = f_{\emptyset,\emptyset^\complement}$.
	\end{enumerate}
\end{lemma}

\begin{proof}
	The statement is true for $i = j = 0$, where $p_{\emptyset, \emptyset} = 2\cdot \vek{1}$ is not a basic function.
	Case~\ref{lem:paired_is_basic:one_leg_empty} is exactly the situation where one of the legs is empty.
	Hence, we may assume that $p$ is Boolean and that both legs are non-empty, i.\,e. that $\max(i,j) \leq \min(k,n-k)$.

	For \enquote{$\Leftarrow$}, we check the identities in Case~\ref{lem:paired_is_basic:k=1} and~\ref{lem:paired_is_basic:i+j=1}, showing that $p_{I,J}$ is basic.
	For \enquote{$\Rightarrow$}, in the case $i = j = k$, we get $\mathcal{P}_{I,J} = \{I,J\} = f_{\emptyset, I\cup J}$ and, in the case $i = j = n-k$, we get $\mathcal{P}_{I,J} = \{I^\complement,J^\complement\} = f_{(I\cup J)^\complement, V}$.%
	\footnote{
		By Lemma~\ref{lem:i=j=k}\ref{lem:i=j=k:basic} and, if necessary, dualization, these representations of $p_{I,J}$ as a basic function are unique.
	}
	For the remaining cases, let $\mathcal{P}_{I,J} = \mathcal{F}_{\!I',J'}$ with $I' \subseteq J'\subseteq V$ disjoint.
	As both legs are non-empty, we have $\#\mathcal{P}_{I,J} \geq 2$, and Lemma~\ref{lem:basic}\ref{lem:basic:generic} yields $I' = \bigcap \mathcal{P}_{I,J}$ and $J' = \bigcup \mathcal{P}_{I,J} = V$.
	Since neither $i=j=k$ nor $i=j=n-k$, by Lemma~\ref{lem:paired:meet_join}, we obtain that $I' = \emptyset$ and $J' = V$.
	Hence, $p_{I,J} = f_{\emptyset, V} = \vek{1}$, so $p_{I,J}$ falls into one of the cases of Lemma~\ref{lem:paired:full}.
	By our assumptions, only Case~\ref{lem:paired:full:i+j=1} of Lemma~\ref{lem:paired:full} is possible, which is $i+j = 1$.
\end{proof}

\begin{lemma}\label{lem:paired_elem_mod}
	Let $I,J\subseteq V$ be disjoint.
	\begin{enumerate}[(a)]
		\item\label{lem:paired_elem_mod:dual} $(p^{(V,k)}_{I,J})^\perp = p^{(V,n-k)}_{I,J}$.
		\item\label{lem:paired_elem_mod:der} For $k\geq 1$ and $x\in V$,
		\[
		\Der_x(p^{(V,k)}_{I,J})
		= \begin{cases}
			f^{(V\setminus\{x\},k-1)}_{I\setminus\{x\},J^\complement} & \text{if }x\in I\text{,} \\[2mm]
			f^{(V\setminus\{x\},k-1)}_{J\setminus\{x\},I^\complement} & \text{if }x\in J\text{,} \\[2mm]
			p^{(V\setminus\{x\},k-1)}_{I,J} & \text{if }x \notin I \cup J\text{.}
		\end{cases}
		\]
		\item\label{lem:paired_elem_mod:res} For $n-k \geq 1$ and $x\in V$,
		\[
		\Res_x(p^{(V,k)}_{I,J})
		= \begin{cases}
			f^{(V\setminus\{x\},k)}_{J,(I \setminus\{x\})^\complement} & \text{if }x\in I\text{,} \\[2mm]
			f^{(V\setminus\{x\},k)}_{I,(J \setminus\{x\})^\complement} & \text{if }x\in J\text{,} \\[2mm]
			p^{(V\setminus\{x\},k)}_{I,J} & \text{if }x \notin I \cup J\text{.}
		\end{cases}
		\]
	\end{enumerate}
	In~\ref{lem:paired_elem_mod:der} and~\ref{lem:paired_elem_mod:res}, set complements are taken relative to the modified ambient space $V \setminus\{x\}$.
\end{lemma}

\begin{proof}
	Use Lemma~\ref{lem:basic_elem_mod}.
\end{proof}

\begin{theorem}\label{thm:paired_determines_I_J}
	Let $g$ be a paired function with $\#g \geq 3$ and $g \neq \vek{1}$.
	Then there is a unique unordered pair $\{I,J\}$ of disjoint sets $I,J\subseteq V$ such that $g = p_{I,J}$.
\end{theorem}

\begin{proof}
	Let $I,J \subseteq V$ be disjoint with $g = p_{I,J}$.
	We will show that, up to swapping, the sets $I$ and $J$ are uniquely determined by $g$.
	We may assume that $g$ is Boolean, as otherwise $g = 2\cdot\vek{1}$ where $I = J = \emptyset$ are uniquely determined.

	If $g$ is basic, then $g$ falls into one of the cases of Lemma~\ref{lem:paired_is_basic}.
	The only case compatible with our assumptions is~\ref{lem:paired_is_basic:one_leg_empty}, such that in the representation $g = p_{I,J}$, one of the legs is empty.%
	\footnote{
		In case~\ref{lem:paired_is_basic:k=1}, we would have $\#g = 2$.
		For that reason, the assumption $\#g \geq 3$ in the statement of the lemma cannot be relaxed to $\#g \geq 2$, even in the case that $g$ is basic.
		As a counterexample, consider $V = \{1,\ldots,n\}$ with $n \geq 4$ and $k=1$.
		Then $\mathcal{P}_{\{1\},\{2\}} = \{\{1\},\{2\}\} = \mathcal{P}_{\emptyset,\{3,4,...,n\}}$ does not uniquely determine $\{I,J\}$.
		The first given representation has an empty leg, while the second one has not.
	}
	So, up to swapping $I$ and $J$, we have $g = f_{I,J^\complement}$.
	Since $\#g \geq 3$ (actually, $\#g \geq 2$ would be sufficient in this step), $I = \bigcap \supp(g)$ and $J = \bigcup \supp(g)$ are uniquely determined by Lemma~\ref{lem:basic}\ref{lem:basic:generic}.

	Now we assume that $g$ is not basic.
	Lemma~\ref{lem:paired_is_basic} yields that $\max(i,j) \leq \min(k,n-k)$ and $i + j \geq 2$, and that, in the case $i = j = 1$, we have $\min(k,n-k) \geq 2$.
	Hence, $\min(k,n-k) \geq 2$ and $n \geq 4$.

	For $x\in V\setminus (I \cup J)$, $\Der_x(g) = p_{I,J^\complement}^{(V\setminus\{x\},k-1)}$ and $\Res_x(g) = p_{I,J^\complement}^{(V\setminus\{x\},k)}$ by Lemma~\ref{lem:paired_elem_mod}.
	By Lemma~\ref{lem:paired_is_basic} and our assumptions, $\Der_x(g)$ is not basic except for $i=j=1$ and $k=2$, and $\Res_x(g)$ is not basic except for $i=j=1$ and $n-k=2$.
	These two conditions cannot be simultaneously true, since otherwise $n=4$, $k=2$, and $i=j=1$, resulting in the contradiction $\#g = 2$.
	On the other hand, for $x\in I\cup J$, the sets $\Der_x(g)$ and $\Res_x(g)$ are always basic, implying that $I \cup J$ is uniquely determined by $g$.

	By $i + j \geq 2$, the set $I \cup J$ is not empty.
	Let $x\in I\cup J$.
	Up to swapping $I$ and $J$, we may assume $x\in I$.
	By $\#\Der_x(g) + \#\Res_x(g) = \#g \geq 3$, we have $\#\Der_x(g) \geq 2$ or $\#\Res_x(g) \geq 2$.
	By Lemma~\ref{lem:paired_elem_mod} and Lemma~\ref{lem:basic}\ref{lem:basic:generic}, we get that in the first case, $\Der_x(g) = f_{I\setminus\{x\},J^\complement}^{(V\setminus\{x\},k-1)}$ uniquely determines $I\setminus\{x\}$ (and thus also $I = (I\setminus\{x\}) \cup \{x\}$) and $J$.
	In the second case, $\Res_x(g) = f_{J,(I\setminus \{x\})}$ uniquely determines $I$ and $J$, too.
\end{proof}

\begin{remark}
	We discuss the situation of the cases not covered by Theorem~\ref{thm:paired_determines_I_J}.
	Paired functions $g$ with $\#g = 0$, $\#g = 1$, or $g = \vek{1}$ are discussed in Lemma~\ref{lem:p_legs}\ref{lem:p_legs:empty}, \ref{lem:paired_size_1} and~\ref{lem:paired:full}, respectively.
	They uniquely determine $\{I,J\}$ only in some extremal border cases, like $\#g = 1$, $n=1$ and $k=0$.

	The paired functions with $\#g = 2$ are classified in Lemma~\ref{lem:paired_size_2}.
	One can show that all ambiguous representations are covered in Lemma~\ref{lem:i=j=k}\ref{lem:i=j=k:paired}, together with the dual situation.
\end{remark}

\section{Proof of the main theorem}\label{sec:main_proof}
To simplify the notation, for all non-negative integers $i,j$ with $i + j \leq n$, we define
\[
	t^{(n,k)}_{i,j} = t_{i,j} = \deg p^{(V,k)}_{I,J}\text{,}
\]
where $I,J$ are disjoint subsets of $V$ of size $\#I = i$ and $\#J = j$.
In doing so, we use the fact that, up to isomorphism, the function $p_{I,J}$ only depends on $i = \#I$ and $j = \#J$, but not on the specific choice of the disjoint sets $I,J\subseteq V$.

\begin{lemma}[Symmetry]\label{lem:t_symm}
	For all non-negative integers $i,j$ with $i + j \leq n$,
	\[
		t^{(n,k)}_{i,j} = t^{(n,k)}_{j,i}\text{.}
	\]
\end{lemma}

\begin{proof}
	See Lemma~\ref{lem:p_symm}.
\end{proof}

For any $I,J \subseteq V$ with $I \cap J = \emptyset$ of size $i = \#I$ and $j = \#J$, we observe that the paired function $p_{I,J} = f_{I,J^\complement} + f_{J,I^\complement}$ is represented by the polynomial
\begin{align}
	\alpha & = \prod_{a\in I} X_a \prod_{b\in J} (1 - X_b) + \prod_{a\in I} (1 - X_a) \prod_{b\in J} X_b \notag \\
	& = ((-1)^j + (-1)^i) \prod_{a\in I \cup J} X_a\; +\; (\text{terms of degree} < i+j) \tag{${\ast}$}\label{eq:basic}
\end{align}
This expression yields two improvements of the elementary bound in Lemma~\ref{lem:elem_bound}.

\begin{lemma}\label{lem:subdiag_odd_upper_bound}
	Let $i,j$ be non-negative integers.
	For $i + j \leq \min(k,n-k)$ and $i + j$ odd, $t^{(n,k)}_{i,j} \leq i+j-1$.
\end{lemma}

\begin{proof}
	For $i + j$ odd, the factor $(-1)^j + (-1)^i$ in Equation~\eqref{eq:basic} vanishes.
\end{proof}

\begin{lemma}\label{lem:subdiag_even}
	Let $i,j$ be non-negative integers.
	For $i + j \leq \min(k,n-k)$ and $i + j$ even, $t^{(n,k)}_{i,j} = i+j$.
\end{lemma}

\begin{proof}
	For $i + j$ even, the term $((-1)^j + (-1)^i) \prod_{a\in I \cup J} X_a$ in Equation~\eqref{eq:basic} represents the function $2f_{I\cup J,\emptyset}$.
	Hence, by Equation~\eqref{eq:basic},
	\[
		\deg(p_{I,J} - 2f_{I\cup J,V}) < i+j\text{.}
	\]
	By Fact~\ref{fct:basic_deg} and $i + j \leq \min(k, n-k)$, we have $\deg f_{I\cup J,V} = i + j$.
	Now Fact~\ref{fct:deg_linear_combination} yields the claim.
\end{proof}

\begin{lemma}\label{lem:t_triad}
	Let $i,j$ be non-negative integers with $i + j < n$.
	Then, among the three numbers $t^{(n,k)}_{i,j}$, $t^{(n,k)}_{i+1,j}$ and $t^{(n,k)}_{i,j+1}$, two are equal, and the third one is either equal or smaller.
\end{lemma}

\begin{proof}
	This follows from Lemma~\ref{lem:paired_pascal_split} and Fact~\ref{fct:deg_linear_combination}.%
	\footnote{
		In the application of Lemma~\ref{lem:paired_pascal_split}, the assumption $i + j < n$ ensures the existence of an $x\notin I \cup J$.
	}
\end{proof}

\subsection{The case $n = 2k$}

We will first resolve the middle layer case $n = 2k$.
Here, we have $k = n-k$.

\begin{lemma}[Monotonicity]\label{lem:monotonicity}
	Let $i,j\in\{0,\ldots,k\}$.
	\begin{enumerate}[(a)]
		\item\label{lem:monotonicity:i} If $i < k$, then $t^{(2k,k)}_{i,j} \leq t^{(2k,k)}_{i+1,j}$.
		\item\label{lem:monotonicity:j} If $j < k$, then $t^{(2k,k)}_{i,j} \leq t^{(2k,k)}_{i,j+1}$.
	\end{enumerate}
\end{lemma}

\begin{proof}
	By the symmetry property in Lemma~\ref{lem:t_symm}, it is enough to show Part~\ref{lem:monotonicity:i}.
	Let $i < k$, and let $I, J\subseteq V$ be disjoint sets of size $\#I = i$ and $\#J = j$.

	First, we show that
	\begin{equation}\label{eq:lem:monotonicity:1}
		(k-i) f_{I,J^\complement} = \sum_{\substack{\bar{I}\in \binom{V}{i+1} \\[0.5mm] I \subseteq \bar{I} \subseteq J^\complement}} f_{\bar{I},J^\complement}\text{.}
	\end{equation}
	Let $K\in\binom{V}{k}$.
	For $K \notin \mathcal{F}_{\!I,J^\complement}$, both sides evaluate to $0$.
	For $K \in \mathcal{F}_{\!I,J^\complement}$, the left hand side evaluates to $k-i$.
	We show that this is also true for the right hand side.
	The sets $\bar{I}$ in the summation on the right hand side are exactly those of the form $\bar{I} = I \cup \{x\}$ with $x\in V\setminus (I \cup J)$.
	For such a set $\bar{I}$, we have $K \in \mathcal{F}_{\!\bar{I},J^\complement} \iff \bar{I} \subseteq K \iff x \in K\setminus I$.
	This condition is true for precisely the $k-i$ elements $x\in K\setminus I$.

	Using $2k = n$, dualization (or a similar counting argument) yields
	\begin{equation}\label{eq:lem:monotonicity:2}
		(k-i) f_{J,I^\complement}
		= \sum_{\substack{\bar{I}\in \binom{V}{i+1} \\[0.5mm] I \subseteq \bar{I} \subseteq J^\complement}} f_{J,\bar{I}^\complement}\text{.}
	\end{equation}
	Addition of Equations~\eqref{eq:lem:monotonicity:1} and~\eqref{eq:lem:monotonicity:2}, followed by a division by $k-i \neq 0$, leads to
	\[
		p_{I,J}
		=
		\frac{1}{k-i}\sum_{\substack{\bar{I}\in \binom{V}{i+1} \\[0.5mm] I \subseteq \bar{I} \subseteq J^\complement}} p_{\bar{I},J}\text{.}
	\]
	After applying the degree, Fact~\ref{fct:deg_linear_combination} yields $t^{(2k,k)}_{i,j} \leq t^{(2k,k)}_{i+1,j}$.
\end{proof}

\begin{remark}
	By our main result, the monotonicity property in Lemma~\ref{lem:monotonicity} is also true without the assumption $n = 2k$, in the range where at most one leg is empty, i.\,e.\ for $\min(i,j) \leq \min(k,n-k)$ and $\max(i,j) \leq \max(k,n-k)$.

	Unfortunately, the above proof does not carry over to the extended situation.
	While the leading factor in Equation~\eqref{eq:lem:monotonicity:1} is still $k-i$, the leading factor in Equation~\eqref{eq:lem:monotonicity:2} would be $n-k-i$, such that the subsequent addition of Equations~\eqref{eq:lem:monotonicity:1} and~\eqref{eq:lem:monotonicity:2} no longer yields the desired result.
	This is the reason why we first treat the case $n=2k$ separately.
\end{remark}

\begin{lemma}\label{n_middle_subdiag_odd}
	Let $i,j$ be non-negative integers.
	For $i + j \leq k$ and $i + j$ odd, $t^{(2k,k)}_{i,j} = i+j-1$.
\end{lemma}

\begin{proof}
	By symmetry, we may assume $i \geq 1$.
	Then
	\[
		t^{(2k,k)}_{i,j}
		\overset{\text{Lem.~\ref{lem:subdiag_odd_upper_bound}}}{\leq} i+j-1
		\overset{\text{Lem.~\ref{lem:subdiag_even}}}{=} t^{(2k,k)}_{i-1,j}
		\overset{\text{Lem.~\ref{lem:monotonicity}}}{\leq} t^{(2k,k)}_{i,j}\text{.}
	\]
\end{proof}

\begin{lemma}\label{lem:n_middle_diag_const}
	Let $s\in\{0,\ldots,2k\}$.
	Then the numbers $t^{(2k,k)}_{i,j}$ are constant over all $i,j\in\{0,\ldots,k\}$ with $i+j = s$.
\end{lemma}

\begin{proof}
	There is nothing to do for $s=0$, so let $s \geq 1$.
	We show that, for any two adjacent positions $(i,j)$ on the diagonal $i + j = s$, the degree is the same.
	Two adjacent positions can be written as $(i'+1,j')$ and $(i', j' + 1)$, where $i'$ and $j'$ are non-negative integers with $i' + j' = s-1$.
	By Lemma~\ref{lem:monotonicity}, $t^{(2k,k)}_{i', j'} \leq \min(t^{(2k,k)}_{i' + 1, j'}, t^{(2k,k)}_{i', j' + 1})$.
	Now by Lemma~\ref{lem:t_triad}, $t^{(2k,k)}_{i' + 1, j'} = t^{(2k,k)}_{i', j' + 1}$.
\end{proof}

\begin{lemma}\label{lem:n_middle_superdiag}
	Let $n = 2k$ and $i,j\in\{0,\ldots,k\}$ with $i + j \geq k$.
	Then
	\[
		t^{(2k,k)}_{i,j} = \begin{cases}
			k & \text{if }k \text{ even,} \\
			k-1 & \text{if }k \text{ odd.}
		\end{cases}
	\]
\end{lemma}

\begin{proof}
	By Lemma~\ref{lem:n_middle_diag_const}, it is enough to consider $i = k$.
	Let $I,J \subseteq V$ be disjoint with $\#I = i = k$ and $\#J = j$.
	Using $n=2k$, Lemma~\ref{lem:paired_size_2} yields $\#p_{I,J} = 2$, both for $j=k$ and $j < k$.
	From $i = k$, we get that $p_{I,J} = \{I,K_2\}$ with $J \subseteq K_2$, so the two blocks of $p_{I,J}$ are disjoint.
	Now Lemma~\ref{lem:i=j=k}\ref{lem:i=j=k:paired} shows that $p_{I,J} = p_{I,\emptyset}$, where we used $n=2k$ again.
	The application of Lemma~\ref{lem:subdiag_even} (for $k$ even) or Lemma~\ref{n_middle_subdiag_odd} (for $k$ odd) concludes the proof.
\end{proof}

Now Lemma~\ref{lem:subdiag_even}, \ref{n_middle_subdiag_odd}, and \ref{lem:n_middle_superdiag} cover all cases of Theorem~\ref{thm:paired_deg} for $n = 2k$.

\subsection{The case $n > 2k$}

We use the concept of derived and residual functions to reduce the case $n > 2k$ to the already completed case $n = 2k$.
For $n > 2k$, we have $\min(k,n-k) = k$ and hence also $t^{(n,k)}_{i,j} \leq k$.

\begin{lemma}\label{lem:paired_deg_der_res}
	Let $n > 2k$ and $i,j\in\{0,\ldots,k\}$.
	\begin{enumerate}[(a)]
		\item If $k \geq 1$, then $t^{(n-1,k-1)}_{i,j} \leq t^{(n,k)}_{i,j}$.\label{lem:paired_deg_der_res:der}
		\item $t^{(n-1,k)}_{i,j} \leq t^{(n,k)}_{i,j}$.\label{lem:paired_deg_der_res:res}
	\end{enumerate}
\end{lemma}

\begin{proof}
	Let $I,J\subseteq V$ be disjoint of size $i = \#I$ and $j = \#J$.
	By $i + j \leq 2k < n$, there exists an $x \in V\setminus (I\cup J)$.

	For $k \geq 1$,
	\begin{multline*}
		t^{(n-1,k-1)}_{i,j}
		= \deg_{V\setminus\{x\}}(p^{(V\setminus\{x\},k-1)}_{I,J})
		\overset{\text{Lem.~\ref{lem:paired_elem_mod}\ref{lem:paired_elem_mod:der}}}{=} \deg_{V\setminus\{x\}}(\Der_x(p^{(V,k)}_{I,J})) \\
		\overset{\text{Fact~\ref{fct:der_res}}}{\leq} \deg_V(p^{(V,k)}_{I,J})
		= t^{(n,k)}_{i,j}\text{.}
	\end{multline*}
	Moreover, using that $n - k > k \geq 0$, we have
	\begin{multline*}
		t^{(n-1,k)}_{i,j}
		= \deg_{V\setminus\{x\}}(p^{(V\setminus\{x\},k)}_{I,J})
		\overset{\text{Lem.~\ref{lem:paired_elem_mod}\ref{lem:paired_elem_mod:res}}}{=} \deg_{V\setminus\{x\}}(\Res_x(p^{(V,k)}_{I,J})) \\
		\overset{\text{Fact~\ref{fct:der_res}}}{\leq} \deg_V(p^{(V,k)}_{I,J})
		= t^{(n,k)}_{i,j}\text{.}
	\end{multline*}
\end{proof}

\begin{lemma}\label{lem:n_large_subdiag_odd}
	Let $n > 2k$ and $i,j$ be non-negative integers with $i + j \leq k$ and $i + j$ odd.
	Then $t^{(n,k)}_{i,j} = i + j - 1$.
\end{lemma}

\begin{proof}
We have
\[
	i+j-1
	\overset{\text{Lem~\ref{n_middle_subdiag_odd}}}{=} t^{(2k,k)}_{i,j}
	\overset{\text{Lem~\ref{lem:paired_deg_der_res}\ref{lem:paired_deg_der_res:res}}}{\leq} t^{(n,k)}_{i,j}
	\overset{\text{Lem~\ref{lem:subdiag_odd_upper_bound}}}{\leq} i + j - 1\text{,}
\]
where Lemma~\ref{lem:paired_deg_der_res}\ref{lem:paired_deg_der_res:res} was applied $n-2k$ times.
\end{proof}

\begin{lemma}\label{lem:n_large_superdiag_even}
	Let $k$ be even, $n > 2k$ and $i,j\in\{0,\ldots,k\}$ with $i + j \geq k$.
	Then $t^{(n,k)}_{i,j} = k$.
\end{lemma}

\begin{proof}
	We have
	\[
		k \overset{\text{Lem.~\ref{lem:n_middle_superdiag}}}{=} t^{(2k,k)}_{i,j}
		\overset{\text{Lem.~\ref{lem:paired_deg_der_res}\ref{lem:paired_deg_der_res:res}}}{\leq} t^{(n,k)}_{i,j}
		\leq k\text{,}
	\]
	where Lemma~\ref{lem:paired_deg_der_res}\ref{lem:paired_deg_der_res:res} was applied $n-2k$ times.
\end{proof}

\begin{lemma}\label{lem:n_large_superdiag_odd}
	Let $k$ be odd, $n > 2k$ and $i,j\in\{0,\ldots,k\}$ with $i + j \geq k$.
	Then $t^{(n,k)}_{i,j} = k$.
\end{lemma}

\begin{proof}
	By $k$ odd, we have $k \geq 1$ and $k-1$ even.
	Furthermore, $2(k-1) < n - 2 < n-1$.
	Hence
	\[
		k
		\overset{\text{Lem.~\ref{lem:n_large_superdiag_even}}}{=} t^{(n-1,k-1)}_{i,j}
		\overset{\text{Lem.~\ref{lem:paired_deg_der_res}\ref{lem:paired_deg_der_res:der}}}{\leq} t^{(n,k)}_{i,j}
		\leq k\text{.}
	\]
\end{proof}

Now Lemma~\ref{lem:subdiag_even}, \ref{lem:n_large_subdiag_odd}, \ref{lem:n_large_superdiag_even}, and \ref{lem:n_large_superdiag_odd} cover all cases of Theorem~\ref{thm:paired_deg} for $n > 2k$.

\subsection{The case $n < 2k$}

\begin{lemma}\label{lem:paired_dual_deg}
	For all non-negative integers $i,j$ with $i + j \leq n$, we have
	\[
		t^{(n,k)}_{i,j} = t^{(n,n-k)}_{i,j}\text{.}
	\]
\end{lemma}

\begin{proof}
Let $I,J\subseteq V$ disjoint of size $\#I = i$ and $\#J = j$.
Then
\[
	t^{(n,k)}_{i,j}
	= \deg(p^{(V,k)}_{I,J})
	\overset{\text{Fact~\ref{fct:dual_deg}}}{=} \deg((p^{(V,k)}_{I,J})^\perp)
	\overset{\text{Lem.~\ref{lem:paired_elem_mod}\ref{lem:paired_elem_mod:dual}}}{=} \deg(p^{(V,n-k)}_{I,J})
	= t^{(n,n-k)}_{i,j}\text{.}
\]
\end{proof}

Since $n < 2k$ implies $n > 2(n-k)$, Lemma~\ref{lem:paired_dual_deg} reduces the case $n < 2k$ to the already settled case $n > 2k$.
The proof of Theorem~\ref{thm:paired_deg} is complete.

\section{An application to Hartman's conjecture}\label{sec:hartman}

In design theory, Hartman's conjecture states that a halving exists as soon as its parameters are admissible~\cite[p.~223]{Hartman-1987-AnnDiscM34:207-224}.
In fact, the conjecture can be reduced to the existence of the \emph{root cases}, which are designs of the parameters considered in the following corollary.

\begin{corollary}\label{cor:hartman}
	Let $a$ be a positive integer, $k = 2^a - 1$, and $D \subseteq\binom{V}{k}$ be a design with the parameters $(k-1)-(2k,k,2^{a-1})$.
	Then $D$ is \emph{anti-complementary}, meaning that for all $K\in\binom{V}{k}$, we have $K \in D \iff K^\complement \notin D$.
\end{corollary}

\begin{proof}
	Let $K\in\binom{V}{k}$.
	Since $k = 2^a - 1$ is odd, by Theorem~\ref{thm:paired_deg}, the set $A \coloneqq p_{K,\emptyset} = \{K,K^\complement\}$ has degree $k-1$.
	Now by Fact~\ref{fct:design_orthogonality}, $\#(A \cap D) = \frac{\lambda_{\max}}{2} \cdot \frac{1}{\lambda_{\max}} \cdot \#A = 1$.
\end{proof}

\begin{remark}
	The result of Corollary~\ref{cor:hartman} can also be derived from the intersection numbers $\alpha_i(S)$ of $D$ with respect to a set $S\in\binom{V}{k}$.
	Equation~\eqref{eq:lambda_i} yields
	\[
		\lambda_i
		= \binom{2k-i}{k-i}/\binom{k+1}{1} \cdot 2^{a-1}
		= \frac{1}{2} \binom{2k-i}{k}\text{.}
	\]
	Using~\cite[Equ.~(5.25)]{Graham-Knuth-Patashnik-1994-ConcreteMathematics2nd},%
	\footnote{
		We apply the formula in the version $\sum_{\substack{i \leq \ell}} (-1)^i \binom{r}{k+i}\binom{\ell-i}{m} = (-1)^k \binom{k+\ell-r}{k+\ell-m}$.
	}
	we compute
	\[
		\sum_{j=0}^{k} (-1)^j \binom{k}{j} \binom{2k-j}{k}
		= (-1)^0 \binom{2k-k}{2k-k} = 1
	\]
	With that preparation, we can evaluate the Köhler equation (Fact~\ref{fct:koehler}) for $i=0$:
	\begin{multline*}
		\alpha_0(S)
		= \sum_{j=0}^{k-1} (-1)^j \binom{k}{j} \lambda_j - \alpha_k(S) \\
		= \frac{1}{2}\left(\sum_{j=0}^k (-1)^j \binom{k}{j} \binom{2k-j}{k} + 1\right) - \alpha_k(S)
		= 1 - \alpha_k(S)\text{.}
	\end{multline*}
	Hence, $\{\alpha_0(S), \alpha_k(S)\} = \{0,1\}$, which means that $D$ is anti-complementary.
\end{remark}

\section{Boolean degree $t$ functions of small size}\label{sec:small_size}
Given non-negative integers $t$ and $k \leq n$, it is natural to ask for the minimum possible size $m_1(n,k,t)$ of a non-zero Boolean function $\binom{V}{k} \to \{0,1\}$ of degree at most $t$, and moreover for a classification of the associated functions.
Dualization (Fact~\ref{fct:dual_deg}) yields
\[
	m_1(n,k,t) = m_1(n,n-k,t)\text{,}
\]
such that, principally, the investigation can be restricted to $2k \leq n$.

The divisibility property in Fact~\ref{fct:divisibility} yields:
\begin{lemma}\label{lem:m1_delta}
	$m_1(n,k,t)$ is a multiple of
	\[
		\Delta = \gcd\left(\tbinom{n-0}{k-0}, \tbinom{n-1}{k-1}, \ldots, \tbinom{n-t}{k-t}\right)\text{.}
	\]
\end{lemma}
The hitherto best known construction in the case $2k \leq n$ is the basic function $f_{T,V}$ for a set $T\in\binom{V}{t}$, also known as a \emph{$t$-pencil}, resulting in:
\begin{lemma}[Pencil bound]
\[
	m_1(n,k,t) \leq \binom{n-t}{\min(k,n-k)-t}\text{.}
\]
\end{lemma}
Furthermore, there is the following lower bound:
\begin{lemma}\label{lem:m1_ge_2}
	For $t < \min(k,n-k)$,
	\[
		m_1(n,k,t) \geq 2\text{.}
	\]
\end{lemma}

\begin{proof}
	Assume that there exists a function $f : \binom{V}{k} \to \{0,1\}$ of degree $t$ and size $1$.
	Then by Fact~\ref{fct:deg_linear_combination}\ref{fct:deg_linear_combination:sum}, all Boolean functions $\binom{V}{k} \to \{0,1\}$ are of degree at most $t$.
	This contradicts the fact that, by $t < \min(k,n-k)$, there exist basic functions $\binom{V}{k} \to \{0,1\}$ of degree $t+1$.
\end{proof}

Clearly, $m_1(n,k,0) = \binom{n}{k}$ (remember that all functions of degree $0$ are constant) and for all $t \geq \min(k,n-k)$, we have $m_1(n,k,t) = 1$.
Moreover, by the classification of the functions of degree $1$ in Fact~\ref{fct:classification_degree_1}, we have $m_1(n,k,1) = \binom{n-1}{\min(k,n-k)-1}$ for $\min(k,n-k) \neq 0$.
In all these cases, $t$-pencils or their duals provide minimal examples, so the pencil bound is sharp.
Remarkably, in certain situations, paired sets of degree $t$ attain or even fall below the pencil bound.

In the following, we assume $2k \leq n$ (i.e. $\min(k,n-k) = k$) and $t\leq k$, which, by the above discussion, is actually not a restriction.

\newpage

\begin{theorem}\label{thm:paired_vs_pencil}
	Let $k \leq 2n$, $t\in\{1,\ldots,k-1\}$, $I,J\subseteq V$ be disjoint of size $i = \#I$ and $j = \#J$ and $p_{I,J} : \binom{V}{k} \to \R$ a paired function of degree $t$.
	Then $p_{I,J}$ is smaller than a $t$-pencil precisely in the following cases:
	\begin{enumerate}[(i)]
		\item\label{thm:paired_vs_pencil:improve_main} $t$ even and $n = 2k$ and $\{i,j\} = \{t+1,0\}$, where $\#p_{I,J} = 2\cdot \binom{2k-t-1}{k}$.
		If additionally $k=t+1$ (where $\#p_{I,J} = 2$), then there are the alternative representations $\{i,j\} = \{t+1,\ell\}$ with $\ell \in \{0,\ldots,t+1\}$.

		\item\label{thm:paired_vs_pencil:improve_superdiagonal} $t$ even and $n = 2k$ and $k = t+1$ and $i,j\in\{1,\ldots,k-1\}$ with $i+j > k$ and $\binom{2k-i-j}{k-i} < \frac{k+1}{2}$, where $\#p_{I,J} = 2\binom{2k-i-j}{k-i}$.

	\end{enumerate}
	Moreover, $p_{I,J}$ has the same size as a $t$-pencil precisely in the following cases:
	\begin{enumerate}[resume*]
		\item\label{thm:paired_vs_pencil:equalize_main} $t$ even and $t \leq k$ and $n=2k + 1$ and $\{i,j\} = \{t+1,0\}$.

		\item\label{thm:paired_vs_pencil:equalize_superdiagonal} $t$ even and $n = 2k$ and $k = t+1$ and $i,j\in\{1,\ldots,k-1\}$ with $i+j \geq k$ and $\binom{2k-i-j}{k-i} = \frac{k+1}{2}$.
	\end{enumerate}
\end{theorem}

\begin{proof}
	We assume that $\#p_{I,J}$ is less or equal the size of a $t$-pencil and, without restriction, $j \leq i$.
	The condition $t\notin\{0,-\infty\}$ is equivalent to $p_{I,J}$ being non-constant.
	Hence, $p_{I,J}$ is non-empty and neither equals $\vek{1}$ nor $2\cdot \vek{1}$.
	In particular, $p_{I,J}$ is Boolean and $\mathcal{P}_{I,J}$ is the disjoint union of two basic sets, and we have $i + j \geq 2$ and thus $i \geq 1$.
	We continue by examining the three possible cases in Theorem~\ref{thm:paired_deg}.

	\paragraph{Case 1.}
	Here, $i+j$ is odd, $i+j \leq k$, and $t = i+j-1$.
	We have $i \geq 2$, since $i = 1$ combined with $j \leq i$ and $i + j \geq 2$ yields $j = 1$, in contradiction to $i + j$ being odd.
	Therefore,
	\[
		\begin{tblr}{colsep=0pt,column{2}={mode=math,colsep=4pt},column{4}={colsep=3pt},column{6}={colsep=4pt},colspec={rcccccc}}
		\#p_{I,J} & = & \dbinom{n-i-j}{k-i} & + & \dbinom{n-i-j}{k-j} & \eqqcolon & a_1 + a_2\text{,} \\
		\text{size of $t$-pencil}
		= \dbinom{n-t}{k-t}
		& = & \dbinom{n-i-j}{k-i-j} & + & \dbinom{n-i-j}{k-i-j+1} & \eqqcolon & b_1 + b_2\text{,}
		\end{tblr}
	\]
	We call the representation $\binom{u}{v}$ of a binomial coefficient \emph{normalized} if $v \leq u-v$.
	For normalized binomial coefficients, we have $\binom{u}{v} \leq \binom{u}{w}$ if and only if $v \leq w$, with equality $\binom{u}{v} = \binom{u}{w}$ if and only if $v = w$.

	We will call $a_1$ (or $a_2$, $b_1$, $b_2$) normalized if the binomial coefficient representation in the above definition is normalized.
	By
	\begin{align*}
		(n - i - j) - (k - i) & = (n-k) - j \geq k - i\text{,} \\
		(n-i-j) - (k-i-j) & = n-k \geq k > k-i-j\quad\text{ and } \\
		(n-i-j) - (k-i-j+1) & = (n-k)-1 \geq k - 1 \overset{i+j\geq 2}{\geq} k-i-j+1\text{,}
	\end{align*}
	the binomial coefficients $a_1$, $b_1$ and $b_2$ are normalized.
	\newpage

	By $k-i \geq k-i-j$, we have $a_1 \geq b_1$.
	Equality $a_1 = a_2$ arises if and only if $k-i = k-i-j$, which is equivalent to $j=0$.

	For the comparison of $a_2$ with $b_2$, we distinguish two cases.
	\begin{itemize}
		\item If $a_2$ is normalized, we have $a_2 > b_2$.

		The reason is that $i \geq 1$ implies $k-j \geq k-i-j+1$, so $a_2 \geq b_2$.
		The equality case $a_2 = b_2$ is equivalent to $i=1$, which is not possible.

		\item Assume that $a_2$ is not normalized.

		Then the representation $a_2 = \binom{n-i-j}{(n-i-j)-(k-j)} = \binom{n-i-j}{n-k-i}$ is normalized.
		Hence $a_2 \leq b_2$ if and only if $n-k-i \leq k-i-j+1$, which is equivalent to $n-2k \leq 1-j$.
		Therefore
		\begin{align*}
			a_2 < b_2 & \iff (n,j) = (2k,0)\quad\text{and} \\
			a_2 = b_2 & \iff (n,j) \in \{(2k,1),(2k+1,0)\}\text{.}
		\end{align*}
		Note that, in all these cases, $a_2$ is indeed not normalized.
	\end{itemize}
	Hence, for $n=2k$ and $j=0$, we have $a_1 = b_1$ and $a_2 < b_2$, leading to the main representation in Case~\ref{thm:paired_vs_pencil:improve_main}.
	In all other cases, $a_1 \geq b_1$ and $a_2 \geq b_2$.
	The equality $a_1 + a_2 = b_1 + b_2$ is equivalent to $n = 2k+1$ and $j=0$, which is Case~\ref{thm:paired_vs_pencil:equalize_main}.

	\paragraph{Case 2.}
	Here, $k$ is odd, $n=2k$, $i + j \geq k$, and $t = k-1$.
	By $t\neq 0$, we have $k \neq 1$ and thus $k \geq 3$.
	We can assume $i + j > k$, as the equality case is already covered by Case~1.
	The size of $p_{I,J}$ is $\#p_{I,J} = 2\binom{2k-i-j}{k-i}$, and the size of the relevant $t$-pencil is $\binom{n-t}{k-t} = \binom{2k - (k-1)}{k - (k-1)} = k+1$.
	For $i=k$, we get $\#p_{I,J} = 2 < k+1$.
	The resulting paired sets are the alternative representations in Case~\ref{thm:paired_vs_pencil:improve_main}.
	For $i < k$, we have $i,j\in\{1,\ldots,k-1\}$, and we get the condition $\binom{2k-i-j}{k-i} \leq \frac{k+1}{2}$.
	Thus, we found Case~\ref{thm:paired_vs_pencil:improve_superdiagonal} and Case~\ref{thm:paired_vs_pencil:equalize_superdiagonal}.

	\paragraph{Case 3.}
	Here, $t = \min(k,i+j)$, and both legs are basic sets of degree $t$.
	Since $p_{I,J}$ has at most the size of a $t$-pencil, which (using $2k \leq n$) is the smallest possible size of a basic set of degree $t$,%
	\footnote{
		The $t$-pencils are unique with this property, except in the case $n = 2k$, where also the duals of $t$-pencils are possible.
	}
	exactly one of the legs is empty (and the second is a $t$-pencil or its dual).
	So, by Lemma~\ref{lem:p_legs}\ref{lem:p_legs:one_empty} and $j \leq i$, we have $k < i$.
	This yields the contradiction $t = \min(k,i+j) = k$.
\end{proof}

\begin{remark}
	We investigate the condition $\binom{2k-i-j}{k-i} = \frac{k+1}{2}$ from Theorem~\ref{thm:paired_vs_pencil}\ref{thm:paired_vs_pencil:equalize_superdiagonal}.
	The substitution of $i$ and $j$ with $a = k-i$ and $b=k-j$ yields the parametrization
	\[
		k = 2\binom{a+b}{a} - 1
		\quad\text{and}\quad i = k-a
		\quad\text{and}\quad j = k-b
	\]
	of the solutions where $a$ and $b$ are any two positive integers.

	Now let us have a closer look at both the inequality and the equality condition from Theorem~\ref{thm:paired_vs_pencil}\ref{thm:paired_vs_pencil:improve_superdiagonal} and~\ref{thm:paired_vs_pencil:equalize_superdiagonal} in two special cases.
	\begin{enumerate}[(a)]
	\item
	For $i = k-1 = t$, the combined condition is $k+1-j \leq \frac{k+1}{2}$, which is equivalent to $j \geq \frac{k+1}{2} = \frac{t}{2} + 1$.
	Hence, we identified the following subcase of Theorem~\ref{thm:paired_vs_pencil}\ref{thm:paired_vs_pencil:improve_superdiagonal} and~\ref{thm:paired_vs_pencil:equalize_superdiagonal}:

	Let $t$ be even, $k = t+1$, and $n = 2k$.
	For $\{i,j\} = \{t,\frac{t}{2} + 1\}$, the paired set $p_{I,J}$ has the same size as a $t$-pencil, and for $\{i,j\} = \{t,\ell\}$ with $\ell\in\{\frac{t}{2} + 2,\ldots, t-1\}$, the size $\#p_{I,J}$ is strictly smaller than that of a $t$-pencil.

	\item
	For $i = k-2 = t-1$, we get the combined condition $\binom{k + 2 - j}{2} \leq \frac{k+1}{2}$, resulting in the quadratic inequality
	\[
		\left(\frac{2k+3}{2} - j\right)^{\!2} \leq k + \frac{5}{4}
	\]
	which is equivalent to
	\[
		j \geq \frac{2k + 3}{2} - \sqrt{k + \frac{5}{4}}\text{.}
	\]
	By the above considerations with $a = k-i = 2$, the equality cases are given by the pairs $(k,j) = b^2 + 3b + 1, b^2 + 2b + 1)$ with a positive integer $b$.
	In terms of $c = b+1$, this is simplified slightly to the expression $(k,j) = (c^2 + c - 1, c^2)$ with an integer $c \geq 2$.
	\end{enumerate}
\end{remark}

\begin{corollary}\label{cor:small_middle_layer_sets}
	Let $t\in\{0,\ldots,k-1\}$ be even.
	Then
	\[
		m_1(2k,k,t) \leq 2\cdot\binom{2k-t-1}{k} = 2\cdot\binom{2k-t-1}{k-t-1}\text{.}
	\]
\end{corollary}

\begin{proof}
	The paired construction with $i=t+1$ and $j=0$ has degree $t$ and size $\binom{2k-t-1}{k}$.
\end{proof}

\begin{remark}
	For $t=0$, the paired construction in the proof of Corollary~\ref{cor:small_middle_layer_sets} is the full set $\binom{V}{k}$, which is also a $0$-pencil.
	As already seen in Theorem~\ref{thm:paired_vs_pencil}\ref{thm:paired_vs_pencil:improve_main}, in all other cases, i.\,e. for $t\in\{1,\ldots,k-1\}$, Corollary~\ref{cor:small_middle_layer_sets} improves the pencil bound.
\end{remark}

\begin{corollary}\label{cor:m1=2}
	For $k$ odd
	\[
		m_1(2k,k,k-1) = 2\text{.}
	\]
\end{corollary}

\begin{proof}
	By Lemma~\ref{lem:m1_ge_2}, $m_1(2k,k,k-1) \geq 2$, and by Corollary~\ref{cor:small_middle_layer_sets}, $m_1(2k,k,k-1) \leq 2$.
\end{proof}

For an assessment of the above bounds, we examined the function $m_1(n,k,t)$ further.
The results are summarized in Tables~\ref{tbl:t2}, \ref{tbl:t3}, and \ref{tbl:t4}.
For given parameters $n, k, t$ as specified in the first three columns, the columns \enquote{\#pencil} and \enquote{\#paired} list the smallest possible size of a pencil (i.\,e., $\#f_{t,0}$) and of a paired function (which is $\#p_{t+1,0}$ for $t$ even, and $\#p_{t,0}$ for $t$ odd), respectively.%
\footnote{
	We have $\#\text{pencil} = \#f_{t,0} = \binom{n-t}{k-t}$, and $\#\text{paired} = \#p_{t+1,0} = \binom{n-t-1}{k} + \binom{n-t-1}{k-t-1}$ when $t$ is even, and $\#\text{paired} = \#p_{t,0} = \binom{n-t}{k} + \binom{n-t}{k-t}$ when $t$ is odd.
	For the symbols $f_{t,0}$ and $p_{t,0}$, see Footnote~\ref{footnote:lazy_p}.
}
In our investigations, we never encountered any instance smaller than both a pencil and a paired construction, so we aimed to establish good lower bounds on $m_1(n,k,t)$ computationally.
We applied the following three methods, listed in order of growing complexity.
\begin{itemize}
	\item In the column \enquote{$\Delta$}, we evaluated the divisibility condition $\Delta \mid m_1(n,k,t)$ from Lemma~\ref{lem:m1_delta}.
	In a few cases, this already determines $m_1(n,k,t)$.
	\item In the column \enquote{LP}, we evaluated the LP bound from Lemma~\ref{lem:lp_bound}, which is a lower bound to $m_1(n,k,t)$.
	\item In case the combination of the above two restrictions still do not close the gap to the size of the best-known construction, we formulated the geometric property \cite[Th.~3.5(iii)]{Kiermaier-Mannaert-Wassermann-2025-JCTSA212:P105979} of sets of degree $t$ as a Diophantine linear equation system and applied the LLL-based solver \texttt{solvediophant}%
	\footnote{
		The solver is accessible at \url{https://github.com/alfredwassermann/solvediophant}.
	}
	\cite{Wassermann-2021-LNCS12757:20-33} to compute the smallest multiple of $\Delta$ for which a solution of that size exists.%
	\footnote{
		Details of this method will be published in a forthcoming article presenting a full classification of all sets of degree $t$ in small feasible cases.
	}
	The result is documented in the column \enquote{\texttt{sd}}.
	A missing entry indicates that applying this method was unnecessary, while a \enquote{?} signifies that the problem turned out to be too large for this approach.
\end{itemize}
The last column summarizes the remaining range or exact value of $m_1(n,k,t)$.

\begin{table}
	\caption{Computational investigation of $m_1$ for $t=2$.}\label{tbl:t2}
	\[
		\begin{tblr}{
			colspec = {rrr@{\hskip 10mm}rr@{\hskip 10mm}rrr@{\hskip 10mm}r},
		}
			\toprule
			n  & k & t & \#\text{pencil} & \#\text{paired} & \Delta & \text{LP} & \text{\texttt{sd}} & m_1(n,k,t) \\
			\midrule
			6  & 3 & 2 &   4 &   2 &  2     & 2.00            &            & 2       \\
			7  & 3 & 2 &   5 &   5 &  5     & 5.00            &            & 5       \\
			8  & 3 & 2 &   6 &  11 &  1     & 6.00            &            & 6       \\
			9  & 3 & 2 &   7 &  21 &  7     & 7.00            &            & 7       \\
			10 & 3 & 2 &   8 &  36 &  4     & 8.00            &            & 8       \\
			11 & 3 & 2 &   9 &  57 &  3     & 9.00            &            & 9       \\
			12 & 3 & 2 &  10 &  85 &  5     & 10.00           &            & 10      \\
			13 & 3 & 2 &  11 & 121 & 11     & 11.00           &            & 11      \\
			\midrule
			8  & 4 & 2 &  15 &  10 &  5     & 6.67            &            & 10      \\
			9  & 4 & 2 &  21 &  21 &  7     & 9.33            & 21         & 21      \\
			10 & 4 & 2 &  28 &  42 & 14     & 14.00           & 28         & 28      \\
			11 & 4 & 2 &  36 &  78 &  6     & 22.00           & 36         & 36      \\
			12 & 4 & 2 &  45 & 135 & 15     & 36.82           &            & 45      \\
			13 & 4 & 2 &  55 & 220 & 55     & 55.00           &            & 55      \\
			\midrule
			10 & 5 & 2 &  56 &  42 & 14     & 18.67           & 42         & 42      \\
			11 & 5 & 2 &  84 &  84 & 42     & 42.00           & 84         & 84      \\
			12 & 5 & 2 & 120 & 162 &  6     & 57.00           & 120        & 120     \\
			13 & 5 & 2 & 165 & 297 & 33     & 77.00           & 165        & 165     \\
			\midrule
			12 & 6 & 2 & 210 & 168 & 42     & 63.00           & 168        & 168     \\
			13 & 6 & 2 & 330 & 330 & 66     & 99.00           & ?          & 99\text{--}330  \\
			14 & 6 & 2 & 495 & 627 & 33     & 165.00          & ?          & 165\text{--}495 \\
			\bottomrule
		\end{tblr}
	\]
\end{table}

\begin{table}
	\caption{Computational investigation of $m_1$ for $t=3$.}\label{tbl:t3}
	\[
		\begin{tblr}{
			colspec = {rrr@{\hskip 10mm}rr@{\hskip 10mm}rrr@{\hskip 10mm}r},
		}
			\toprule
			n  & k & t & \#\text{pencil} & \#\text{paired} & \Delta & \text{LP} & \text{\texttt{sd}} & m_1(n,k,t) \\
			\midrule
			8  & 4 & 3 &   5 &   10 &  5      & 5.00            &            & 5          \\
			9  & 4 & 3 &   6 &   21 &  1      & 6.00            &            & 6          \\
			10 & 4 & 3 &   7 &   42 &  7      & 7.00            &            & 7          \\
			11 & 4 & 3 &   8 &   78 &  2      & 8.00            &            & 8          \\
			12 & 4 & 3 &   9 &  135 &  3      & 9.00            &            & 9          \\
			13 & 4 & 3 &  10 &  220 &  5      & 10.00           &            & 10         \\
			\midrule
			10 & 5 & 3 &  21 &   42 &  7      & 9.33            & 21         & 21         \\
			11 & 5 & 3 &  28 &   84 & 14      & 14.00           & 28         & 28         \\
			12 & 5 & 3 &  36 &  162 &  6      & 22.00           & 36         & 36         \\
			13 & 5 & 3 &  45 &  297 &  3      & 35.00           & ?          & 36\text{--}45      \\
			\midrule
			12 & 6 & 3 &  84 &  168 & 42      & 42.00           & ?          & 42\text{--}84      \\
			13 & 6 & 3 & 120 &  330 &  6      & 57.00           & ?          & 60\text{--}120     \\
			14 & 6 & 3 & 165 &  627 & 33      & 77.00           & ?          & 99\text{--}165     \\
			15 & 6 & 3 & 220 & 1144 & 11      & 104.50          & ?          & 110\text{--}220    \\
			\bottomrule
		\end{tblr}
	\]
\end{table}

\begin{table}
	\caption{Computational investigation of $m_1$ for $t=4$.}\label{tbl:t4}
	\[
		\begin{tblr}{
			colspec = {rrr@{\hskip 10mm}rr@{\hskip 10mm}rrr@{\hskip 10mm}r},
		}
			\toprule
			n  & k & t & \#\text{pencil} & \#\text{paired} & \Delta & \text{LP} & \text{\texttt{sd}} & m_1(n,k,t) \\
			\midrule
			10 & 5 & 4 &  6 &   2 & 1      &  2.00           &            & 2          \\
			11 & 5 & 4 &  7 &   7 & 7      &  7.00           &            & 7          \\
			12 & 5 & 4 &  8 &  22 & 2      &  8.00           &            & 8          \\
			13 & 5 & 4 &  9 &  57 & 3      &  9.00           &            & 9          \\
			\midrule
			12 & 6 & 4 & 28 &  14 & 14     &  8.40           &            & 14         \\
			13 & 6 & 4 & 36 &  36 & 6      & 10.80           & ?          & 12\text{--}36      \\
			14 & 6 & 4 & 45 &  93 & 3      & 16.20           & ?          & 18\text{--}45      \\
			15 & 6 & 4 & 55 & 220 & 11     & 30.56           & ?          & 33\text{--}55      \\
			\bottomrule
		\end{tblr}
	\]
\end{table}

Based on these findings, we conclude this article with the following conjecture.
\begin{conjecture}\label{conj:pencil_paired_are_smallest}
	Let $n,k,t$ be non-negative integers with $k \leq n$ and $t\leq\min(k,n-k)$.
	Then
	\[
		m_1(n,k,t) =
		\begin{cases}
			2\cdot\displaystyle\binom{2k-t-1}{k} & \text{for }n=2k\text{ and }t\text{ even and }t \neq k\text{,} \\[5mm]
			\displaystyle\binom{n-t}{\min(k,n-k)-t} & \text{otherwise.}
		\end{cases}
	\]
\end{conjecture}

\vspace{-3mm}

\begin{remark}
	In the range $2k \leq n$, we summarize all constructions known to us that attain the conjectured value of $m_1(n,k,t)$.
	\begin{itemize}
		\item For $n = 2k$ and $t$ even: $p^{(2k,k)}_{t+1,0}$.%
		\footnote{\label{footnote:lazy_p}
			The expression $p^{(2k,k)}_{t,0}$ is used for simplicity and should be read as $p^{(2k,k)}_{T,\emptyset}$ with $T \subseteq V$ and $\#I = t$.
			An analogous remark applies to similar symbols like $f^{(2k,k)}_{t,0}$.
		}
		\item For $n = 2k+1$ and $t$ even: $f^{(2k+1,k)}_{t,0}$ and $p^{(2k+1,k)}_{t+1,0}$.
		Note that, for $t = 0$, both constructions coincide.
		\item For $n = 2k$ and $t$ odd: $f^{(2k,k)}_{t,0}$ and its dual $f^{(2k,k)}_{0,t}$.
		Furthermore, the function $p^{(2k-1,k-1)}_{t,0} \cdot X_{2k}$ and its dual, which is $p^{(2k-1,k)}_{t,0}\cdot (1-X_{2k})$.
		For $t = 1$, the latter pair of functions coincides with the former.

		For $n=8$, $k=4$, and $t=3$, there is also the characteristic function of
		\[
			\{1234, 1235, 1245, 1678, 2678\}
		\]
		and its dual.%
		\footnote{
			For $k = t+1$ and larger values of $n$, we expect the existence of further examples.
		}
		\item In all other cases: Only the pencil $f^{(2k,k)}_{t,0}$.
	\end{itemize}
\end{remark}
\newpage

\section*{Acknowledgements}
The research of the authors was supported by the Research Network Coding Theory and Cryptography (W0.010.17N) of the Research Foundation – Flanders (FWO).

\printbibliography

\end{document}